\documentclass[reqno]{amsart}
\usepackage{amsmath,amsthm,amsfonts,amssymb,bm }
\usepackage[single]{accents}
\usepackage{cite}
\usepackage{fancyhdr,amscd}
\usepackage{anysize}
\usepackage{graphicx,epsfig}
\usepackage{amsthm,latexsym,amsmath,amssymb}
\usepackage[perpage,symbol*]{footmisc}
\usepackage[colorlinks,linkcolor=black,citecolor=black]{hyperref}
\usepackage{xcolor}
\usepackage{mathrsfs}
\usepackage{cite}
\usepackage{chngcntr}
\numberwithin{figure}{section}
\usepackage{lineno}

\usepackage{indentfirst}
 \usepackage[titletoc]{appendix}
\graphicspath{{figures/}}
\newcommand{\ignore}[1]{}
\numberwithin{equation}{section}
\newcommand{\R}{{\mathbb R}}
\newcommand{\pa}{\partial}

\newtheorem{theorem}{\textbf Theorem}[section]
\newtheorem{lem}{\textbf Lemma}[section]
\newtheorem{rem}{\textbf Remark}[section]

\newtheorem{prop}{\textbf Proposition}[section]

\newtheorem{defin}{\textbf Definition}[section]
\newtheorem{example}{\textbf Example}[section]

\newcounter{remark}
{\par \stepcounter{remark} {\it Remark
\arabic{section}.\arabic{remark}.}~}{\rm  \par}
\usepackage{amsmath}
\allowdisplaybreaks[4]

\allowdisplaybreaks \voffset=-0.2in \numberwithin{equation}{section}

\begin{document}
\begin{titlepage}
\title{\bf An energy stable finite element scheme for the three-component Cahn-Hilliard-type model for macromolecular microsphere composite hydrogels}
\end{titlepage}


\author[M.Yuan and W. Cheng and C.Wang and S.M. Wise and Z.Zhang]{Maoqin Yuan, Wenbin Chen, Cheng Wang, Steven M. Wise and Zhengru Zhang}

\address{School of Mathematical Sciences,
Beijing Normal University, Beijing 100875, P.R. China (\email{mqyuan@mail.bnu.edu.cn})}

\address{Shanghai Key Laboratory of Mathematics for Nonlinear Sciences, School of Mathematical Sciences; Fudan University, Shanghai, China 200433 (\email{wbchen@fudan.edu.cn})}

\address{Department of Mathematics, The University of Massachusetts, North Dartmouth, MA 02747, USA (\email{cwang1@umassd.edu})}

\address{Department of Mathematics, The University of Tennessee, Knoxville, TN 37996, USA (\email{swise1@utk.edu})}

\address{School of Mathematical Sciences,
Beijing Normal University and Laboratory of Mathematics and Complex Systems,
Ministry of Education, Beijing 100875, P.R. China (\email{zrzhang@bnu.edu.cn})}

\maketitle

\begin{abstract}
In this article, we present and analyze a finite element numerical scheme for a three-component macromolecular microsphere composite (MMC) hydrogel model, which takes the form of a ternary Cahn-Hilliard-type equation with Flory-Huggins-deGennes energy potential. The numerical approach is based on a convex-concave decomposition of the energy functional in multi-phase space, in which the logarithmic and the nonlinear surface diffusion terms are treated implicitly, while the concave expansive linear terms are explicitly updated. A mass lumped finite element spatial approximation is applied, to ensure the positivity of the phase variables. In turn, a positivity-preserving property can be theoretically justified for the proposed fully discrete numerical scheme. In addition, unconditional energy stability is established as well, which comes from the convexity analysis. Several numerical simulations are carried out to verify the accuracy and positivity-preserving property of the proposed scheme.
\end{abstract}

\medskip

\medskip


\medskip



	\section{Introduction}
A hydrogel is a network of cross-linked hydrophilic polymers chains. They absorb water and can swell to many times their original size. Hydrogels, which can act as solids or liquids in various settings, are versatile materials that have led to extensive industrial and biomedical applications~\cite{drury_hydrogels_2003, edlund_barrier_2010, johnson_hydrogels_2010}. Macromolecular microsphere composite (MMC) hydrogel, which was originally synthesized by Huang et al.~in 2007,  possesses a unique well-defined network structure and very high mechanical strength. This is due to the highly specialized chemical grafting of the entangled polymer chains, in comparison with traditional hydrogels \cite{huang_novel_2007}. MMC hydrogels have been widely applied in both biomedical and industrial areas, such as in drug delivery~\cite{xiao_monodispersed_2005}, artificial tissues~\cite{Wang_Advances_2008,curk_rational_2016}, et cetera. The formation process of MMC hydrogel has been described in detail in \cite{He_Nanoparticles_2011, huang_novel_2007, zhai_investigation_2013}.

Computational and experimental studies are needed to reveal the complicated properties of MMC hydrogels. Studies must include the investigation of the parameter space related to their production and processing, in order to engineer their individual effects. Furthermore, must be explored and refined to validate their predictions. For example, Zhai et al.~\cite{zhai_investigation_2013} developed a reticular free energy for MMCs, under certain assumptions, most particularly, that the number of graft chains around a macromolecular microsphere (MMS) is proportional to the perimeter. Based on the time-dependent Ginzburg-Landau (TDGL) mesoscale simulation method, a two-component model, appropriately named the MMC-TDGL equation, was developed to understand the time evolution of MMC hydrogel structure in \cite{zhai_investigation_2013}. This continuum scale model was designed to simulate phase transitions in MMC hydrogels. Li et al.~\cite{li_phase_2015} added a stochastic term in the binary MMC-TDGL equation to consider how random physical fluctuations modify the dynamics. Recently, the reticular free energy was reconstructed in~\cite{ji2021modeling}, and shown to be consistent with the network structures of the MMC hydrogels. Based on the Boltzmann entropy theorem, the Flory-Huggins lattice theory and assuming TDGL dynamics,  a three-component MMC-TDGL model can be constructed. The MMC and polymer chains are no longer considered as a whole in this model, making it more consistent with experiments.

It is widely known that phase-field models satisfy certain properties, such as energy decay, mass conservation, and positivity preservation. These properties represent important physical features and are also essential for mathematical analysis and consistent numerical simulation. During the past several decades, there have been many works devoted to designing various kinds of numerical methods to satisfy these properties, especially for Allen-Cahn and Cahn-Hilliard-type equations. See, for example,~\cite{copetti_numerical_1992,wise_energy-stable_2009,cherfils_cahn-hilliard_2011,yang_linear_2016,yang_numerical_2017} and the references therein. Zhai et al.~\cite{zhai_investigation_2013} constructed a spectral-type numerical method to approximate the solutions of the binary MMC-TDGL equations. Li et al.~used a semi-implicit scheme for binary MMC-TDGL equation in \cite{li_phase_2015}, while there was no discussion of any stability condition. Subsequently, a convex splitting method was presented in~\cite{li_unconditionally_2016}, and energy stability was proven for the numerical solution of the phase variable. Liao et al.~applied an adaptive time step strategy to improve computational efficiency in \cite{liao_energy_2017} and Dong et al.~\cite{dong2019, dong2020a} presented the theoretical analysis for the first and second-order energy stable schemes. A stabilized method was also used to solve the binary system by Xu et al.~in \cite{xu_efficient_2019}, though a theoretical justification of the stabilizing parameters (for energy decay) has not been established. Other related works could be found in~\cite{Gong2019, Yang2019}, etc.   

Although there have been many works on the multi-component Cahn-Hilliard flow~\cite{boyer_study_2006, boyer_numerical_2011, ChenTernaryCH, yang_numerical_2017}, addressing polynomial-type energy potentials, the numerical study of ternary MMC-TDGL equations is still in the preliminary stages. First, it has always been a key difficulty to design a numerical scheme satisfying the physical properties. Furthermore, it is highly challenging to prove the positivity-preserving property for the logarithmic terms, since the fourth-order partial differential equations fail to satisfy a maximum principle. In \cite{copetti_numerical_1992}, a finite element scheme was proposed based on the backward Euler approximation for the Cahn-Hilliard equation with logarithmic free energy, and the positivity-preserving property of the numerical solution was proven under a constraint on the time step. In a more recent work \cite{chen19a}, the authors presented a finite difference scheme based on the convex-concave decomposition of the free energy with logarithmic potential and established a theoretical justification of the positivity-preserving property, regardless of time step size. This improvement is based on the following fact: \textit{the singular nature of the logarithmic term around the pure-phase values prevents the numerical solution from reaching these singular values, so that the numerical scheme is always well-defined as long as the numerical solution stays similarly bounded at the previous time step}. Moreover, similar ideas have been applied in \cite{dong2019, dong2020a} to analyze the binary MMC-TDGL equation. Also see the related works of other gradient models with singular energy potential, such as the Poisson-Nernst-Planck system~\cite{LiuC2021a, Qian2021}, the reaction-diffusion system in the energetic variational formulation~\cite{LiuC2021b}, liquid film droplet model~\cite{ZhangJ2021}, etc.  

In this article, we aim to analyze the ternary MMC-TDGL system and obtain the theoretical justification of both the positivity-preserving property and the energy stability. To this end, the key ingredient is an application of the convex-concave decomposition of the physical energy, with respect to the multi-phase variables. In fact, the convex splitting method has been extensively applied to a variety of gradient flow models~\cite{baskaran13a, baskaran13b, chen12, chen16, diegel16, diegel17, fengW18a, gu_energy-stable_2014, guan14b, guan14a, guo16, liuY17, shen12, wang10a, wang11a, wise10, wise_energy-stable_2009, yan18}, for both first and second-order temporally accurate versions.  Meanwhile, most of these existing works have focused on polynomial free energy potentials. The extension to singular Flory-Huggins-type energy potentials turns out to be highly challenging. In addition, the appearance of the highly nonlinear and singular deGennes gradient energy terms makes the whole system even more difficult. To overcome these subtle difficulties, we make use of a convex-concave decomposition of the physical energy in the ternary MMC-TDGL system, reported in a recent work~\cite{Dong2020b}.

In more details, the logarithmic terms and the highly nonlinear gradient energy terms are placed in the convex part, while the expansive terms are put in the concave part, based on careful convexity analyses. In turn, the convex splitting approach leads to a uniquely solvable, positivity-preserving and energy stable numerical scheme. The finite difference approximation was reported in~\cite{Dong2020b}, and its direct application to the finite element method is not available, due to the difficulty to ensure the point-wise positivity of the numerical solution in the standard FEM method. In our work, a mixed FEM method is applied to the ternary MMC-TDGL system to facilitate the numerical implementation of the fourth order parabolic equations. It is well-known that the standard conforming FEM fails to satisfy the discrete maximum principle due to the non-diagonal mass matrix. As a result, a lumped mass FEM was chosen instead, so as to diagonalize the mass matrix. The diagonal elements are the row sums of the original mass matrix~\cite{thomee_galerkin_2006}. In comparison with the finite difference method, the FEM allows for flexible, adaptive meshes and is often easier to analyze.


This paper is organized as follows. In Section~\ref{sec:model}, we briefly review the mathematical model of  three-component phase transitions in  MMC hydrogels. In Section~\ref{sec:numerical scheme}, we present the numerical scheme using the mass lumped finite element method. The detailed proof for the positivity-preserving property of the numerical solution is provided in Section~\ref{sec:positivity}, and the energy stability analysis is established in Section~\ref{sec:energy stability}. In Section~\ref{sec:numerical results}, the numerical simulations are presented to verify the theoretical results. Finally, some concluding remarks are given in Section~\ref{sec:conclusion}.

\section{Three-component MMC-TDGL system} \label{sec:model}

Given an open bounded, connected domain $\Omega \subset \R ^2$ with a Lipschitz smooth boundary $\partial\Omega$, we recall the derivation of the diffuse interface describing the phase transitions of MMC hydrogels. 
It is worth mentioning that the ternary system is made of water, macromolecular microsphere, and polymer chain. Usually, the composition of the mixture is described at each point by the  concentration value of 
one of the constituents in the mixture. Thus we denote the concentration of the macromolecular microsphere in the ternary system by the order parameter $\phi_{1}$, the polymer chain by $\phi_{2}$ and the solvent molecules by $\phi_{3}$. The value of the three order parameters are located between $0$ and $1$, where three phases vary rapidly but smoothly across the interface. And also, these three unknowns are linked through the hyperplain link relationship $\phi_{1}+\phi_{2}+\phi_{3}=1$. Due to the mass conservation constraint, we denote $\phi_3 = 1 - \phi_1 - \phi_2$ throughout the rest of this article, for simplicity of presentation.

The Flory-Huggins reticular free energy takes a form of $f(\phi_1,\phi_2)$. Moreover, the evolution of the system is driven by the minimization of a free energy under the constraint of mass conservation of each phase. The Ginzburg-Landau type energy functional $F(\phi_{1},\phi_{2})$ takes the following form
\begin{equation}\label{ginzburg-landau}
\begin{aligned}
 F \left( \phi_1 , \phi_2  \right)  =
\int_\Omega  f(\phi_1,\phi_2) +  K(\phi_1,\phi_2) d \mathbf {x},
\end{aligned}
\end{equation}
where 
$f(\phi_1,\phi_2)$ contains the mixing entropy $S(\phi_{1},\phi_{2})$ and the mixing enthalpy $H(\phi_{1},\phi_{2})$, i.e., $f(\phi_1,\phi_2)= S(\phi_1,\phi_2) + H(\phi_1,\phi_2)$. The expression of $S(\phi_1,\phi_2), H(\phi_1,\phi_2)$, as well as $K(\phi_{1},\phi_{2})$, can be written as follows
\begin{flalign}\label{SHK}
S(\phi_1,\phi_2) &=\frac { \phi_1 } { \gamma } \ln \left( \frac { \alpha \phi_1 } { \gamma } \right) + \frac { \phi_2 } { N } \ln \left( \frac { \beta \phi_2 } { N } \right)
+ \left( 1 - \phi_1 - \phi_2 \right) \ln \left( 1 - \phi _ {1 } - \phi_2 \right),
\nonumber\\
H(\phi_1,\phi_2) &=\chi_{12} \phi_1 \phi_2 + \chi_{13} \phi_1 \left( 1 - \phi_1 - \phi_2 \right) + \chi_{23} \phi_2 \left( 1 - \phi_1 - \phi_2 \right),\\
K(\phi_1,\phi_2) &=\frac { a_1^2 } { 36 \phi_1 } | \nabla \phi_1 |^2
+ \frac { a_2^2 } { 36 \phi_2 } | \nabla \phi_2 |^2
+ \frac { a_3^2 } { 36 \left( 1 - \phi_1 - \phi_2 \right) } | \nabla \left( 1 - \phi_1 - \phi_2 \right) |^2,\nonumber
\end{flalign}
in which the parameter $\gamma$ is the relative volume of one macromolecular microsphere, $N$ is the degree of polymerization of the polymer chains. The parameters $\alpha$ and $\beta$ are determined by the formulas $\alpha=\pi(\sqrt{\gamma/\pi}+N/2)^2, \beta=\alpha/\sqrt{\pi N}$, dependent on $\gamma$ and $N$; see more detailed derivations of the model in~\cite{zhai_investigation_2013}. In fact, the Flory-Huggins energy density takes a form of $\phi_i \ln \phi_i$ for each species concentration, combined with the interaction energy density $\phi_i \phi_j$~\cite{Flory1953}. The constants $\chi_{12}, \chi_{13}$, and $\chi_{23}$ are the Flory-Huggins interaction parameters between macromolecular microspheres and polymer chain, macromolecular microspheres and solvent, and polymer chain and solvent, respectively. In addition, the deGennes diffusive coefficient, $\kappa (\phi_i) = \frac{a_i^2}{36 \phi_i}$, depends on the corresponding phase variables. This diffusion process was proposed by physicist P.G. deGennes~\cite{deGennes1980} for the binary Cahn-Hilliard flow, in which the phase variables could be simplified as $\phi_1 = \phi$, $\phi_2 = 1 - \phi$, so that the combined diffusion coefficient and surface diffusion energy density becomes   
$$
\begin{array}{l} 
  \displaystyle \vspace{.05in}   
  \kappa (\phi) = \frac{a^2}{36 \phi_1} + \frac{a^2}{36 \phi_2 }  
  = \frac{a^2}{36 \phi} + \frac{a^2}{36( 1 - \phi ) }  
  = \frac{a^2}{36 \phi (1 - \phi)} , \quad | \nabla \phi_1 | = | \nabla \phi_2 | = | \nabla \phi | ,  \\
  \displaystyle \vspace{.05in}  
  K (\phi) = \frac{a^2}{36 \phi_1} | \nabla \phi_1 |^2 +  \frac{a^2}{36 \phi_2} | \nabla \phi_2 |^2 
  = \kappa (\phi) | \nabla \phi |^2 = \frac{a^2}{36 \phi (1 - \phi)} | \nabla \phi |^2 . 
\end{array}    
$$
An extension to the ternary gradient flow is natural. Such a nonlinear diffusive coefficient has been an essential difficulty for the MMC-TDGL model; see the related analysis in~\cite{dong2019, dong2020a}. Here $a_{i}$ is the statistical segment length of the  $i^{\rm th}$ component, $i=1, 2, 3$. By a simple computation, the variational derivatives of the free energy function $F ( \phi_1 ( \mathbf { x } , t ) , \phi_2 ( \mathbf { x } , t ))$ with respect to $\phi_{1}$ and $\phi_{2}$ are found to be
\begin{flalign}
\frac { \delta F \left( \phi_1 , \phi_2 \right) } { \delta \phi_1 } &= \frac {\partial S \left( \phi_1 , \phi_2 \right) } { \partial \phi_1 } - \frac { a_1^2 | \nabla \phi_1 |^2 } { 36 \phi_1^2 } -\nabla \cdot\left( \frac { a_1^2 \nabla  \phi_1 } { 18 \phi_1 }\right)+ \frac { a_3^2 | \nabla \left( 1 - \phi_1 - \phi_2 \right) |^2 } { 36 \left( 1 - \phi_1 - \phi_2 \right)^2 }
\label{DF}\\
\nonumber& \quad + \nabla \cdot\left( \frac { a_3^2 \nabla \left( 1 - \phi_1 - \phi_2 \right) } { 18 \left( 1 - \phi_1 - \phi_2 \right) }\right)
- \frac {\partial H \left( \phi_1 , \phi_2 \right) } { \partial \phi_1 } ,\\
\frac { \delta F \left(\phi_1 , \phi_2 \right)  } { \delta \phi_2 } &=\frac {\partial S \left( \phi_1 , \phi_2 \right) } { \partial \phi_2 } - \frac { a_2^2 | \nabla \phi_2 |^2 } { 36 \phi_2^2 } -\nabla \cdot\left(  \frac { a_2^2 \nabla  \phi_2 } { 18 \phi_2 }\right)+ \frac { a_3^2 | \nabla \left( 1 - \phi_1 - \phi_2 \right) |^2 } { 36 \left( 1 - \phi_1 - \phi_2 \right)^2 }
\label{DF2}\\
\nonumber& \quad + \nabla \cdot\left( \frac { a_3^2 \nabla \left( 1 - \phi_1 - \phi_2 \right) } { 18 \left( 1 - \phi_1 - \phi_2 \right) }\right)-\frac {\partial H \left( \phi_1 , \phi_2 \right) } { \partial \phi_2 }, 
\end{flalign}
where
	\begin{flalign}
\frac {\partial S } { \partial \phi_1 } &= \frac { 1 } { \gamma }\ln \left( \frac { \alpha \phi_1 } { \gamma } \right)
+ \frac { 1 } { \gamma } - 1 - \ln \left( 1 - \phi_1 - \phi_2 \right),\quad \frac {\partial S  } { \partial \phi_2 } = \frac { 1 } { N }\ln \left( \frac { \beta \phi_2 } { N } \right)
+ \frac { 1 } { N } - 1 - \ln \left( 1 - \phi_1 - \phi_2 \right),\nonumber\\
\frac {\partial H  } { \partial \phi_1 } &=
- 2 \chi _ { 13 } \phi_1
+ \left( \chi_{12} - \chi _ { 13 } - \chi _ { 23 } \right) \phi_2
+ \chi _ { 13 },\quad
\frac {\partial H } { \partial \phi_2 } =
- 2 \chi _ { 23 } \phi_2 + \left( \chi_{12} - \chi _ { 13 } - \chi _ { 23 } \right) \phi_1 + \chi _ { 23 }. \nonumber
\end{flalign}

To simulate the traditional hydrogels, the time-dependent Ginzburg-Landau (TDGL) mesoscopic model is widely used to describe the phase transitions of a multi-component polymer blend. Once this energy $F$ is defined, we can formulate the time evolution of the three-component MMC hydrogels system for the conserved Cahn-Hilliard equations:
	\begin{flalign}
	\frac { \pa \phi_1} { \pa t }
	&= D _ { 1 } \Delta \frac { \delta F \left( \phi_1 , \phi_2 \right) } { \delta \phi_1 },\label{MMC-equations}
	\\
	\frac { \pa \phi_2 } { \pa t } &= D _ { 2 } \Delta \frac { \delta F \left(\phi_1  , \phi_2  \right) } { \delta \phi_2},\label{MMC-equations2}
	\end{flalign}
where $D_i= k_B \theta M_i$ are the diffusion coefficients, $k_B$ is the Boltzmann constant, $\theta$ is the temperature, and $M_i > 0$ stand for the mobility of the $i^{\rm th}$ component, $i=1, 2$. For simplicity, we select $\Omega=(0,L)^2$, and consider $L$-periodic boundary condition for this model. However, the finite element method can be extended to a wider class of regions and Neumann boundary conditions could also be used.

The Cahn-Hilliard system has the important feature that the phase variables, $\phi_{1}$ and $\phi_{2}$, are mass-conservative. Integrating \eqref{MMC-equations} and \eqref{MMC-equations2} over $\Omega=(0 , L)^2$, we obtain
\begin{equation}
\begin{aligned}
\frac{d}{dt}\int_{\Omega}\phi_{1}d\mathbf {x} = & \int_{\Omega}\frac{\pa\phi_{1}}{\pa t}d\mathbf {x}=D_{1}\int_{\pa \Omega}\nabla \frac{\delta F}{\delta\phi_{1}}\cdot \boldsymbol{n}ds,\\
\frac{d}{dt}\int_{\Omega}\phi_{2}d\mathbf {x} = & \int_{\Omega}\frac{\pa\phi_{2}}{\pa t}d\mathbf {x}=D_{2}\int_{\pa \Omega}\nabla \frac{\delta F}{\delta\phi_{2}}\cdot \boldsymbol{n}ds .
\end{aligned}
\end{equation}
Notice that $\delta F/\delta \phi_{1}$ and $\delta F/\delta\phi_{2}$ in \eqref{DF} and \eqref{DF2} are both $L$-periodic with respect to $x$ and $y$, 	so that integration on the boundary vanishes, which implies
\begin{equation}
\begin{aligned}
\int_{\Omega}\phi_{1}(\mathbf {x},t)d\mathbf {x} =\int_{\Omega}\phi_{1}(\mathbf {x},0)d\mathbf {x}, \quad 
\int_{\Omega}\phi_{2}(\mathbf {x},t)d\mathbf {x} =\int_{\Omega}\phi_{2}(\mathbf {x},0)d\mathbf {x}, \quad \forall t>0 . 
\end{aligned}
\end{equation}
As a consequence, $\phi_{3}=1-\phi_{1}-\phi_{2}$  also satisfies the mass-conversation property. 

Meanwhile, the most distinguished difficulty for the Cahn-Hilliard equation with logarithmic Flory Huggins energy potential and deGennes diffusive coefficients is associated with the singularity as the value of $\phi$ approaches the limit value $0$. In fact, for the binary Cahn-Hilliard flow, the positivity property,  i.e., $0 < \phi_1, \phi_2$, has been established at the PDE analysis level in~\cite{abels07, debussche95, elliott96b, miranville04}. As a further development, the phase separation has also been justified for the 1-D and 2-D equations at a theoretical level, i.e., a uniform distance between the phase variable and the singular limit values has been proved, and such a distance only depends on the surface diffusion coefficient and expansive parameter, as well as the initial data. For the ternary MMC TDGL model, a similar positivity estimate is expected to be valid for the exact PDE solution, i.e., $0 < \phi_i$, ($1 \le i \le 3$), and a uniform separation property is also expected to be valid for 2-D flow; more technical details have to be involved for this model.    

In terms of the energy stability, by multiplying~\eqref{MMC-equations} with $\delta F/\delta\phi_{1}$ and \eqref{MMC-equations2} with $\delta F/\delta \phi_{2}$, respectively, and integrating it over $\Omega$,  using Green's formula and the periodic boundary conditions, one obtains
\begin{equation}
\begin{aligned}
\frac{d F}{dt}&=\int_{\Omega}\frac{\delta F}{\delta \phi_{1}}\frac{\pa \phi_{1}}{\pa t} d\mathbf {x} +\int_{\Omega}\frac{\delta F}{\delta \phi_{2}}\frac{\pa \phi_{2}}{\pa t} d\mathbf {x}\\
&=-D_{1}\int_{\Omega}|\nabla \frac{\delta F}{\delta \phi_{1}}|^2d\mathbf {x}-D_{2}\int_{\Omega}|\nabla \frac{\delta F}{\delta \phi_{2}}|^2d\mathbf {x} \leq 0,
\end{aligned}
\end{equation}
which indicates that the energy $F(\phi_{1},\phi_{2})$ is a decreasing function of time.

\section{The fully discrete finite element scheme} \label{sec:numerical scheme}

The standard notation for the norms is used, in their respective function spaces. In particular, we denote the standard norms for the Sobolev spaces $ W^{m,p}(\Omega)$ by $\|\cdot\|_{m,p}$, and repleace $\|\cdot\|_{0,p}$ by $\|\cdot\|_p$, $\|\cdot\|_{0,2}=\|\cdot\|_2$ by $\|\cdot\|$, and $\|\cdot\|_{q,2}$ by $\|\cdot\|_{H^q}$. Let $C^{\infty}_{per}(\Omega)$ be the set of all restrictions onto $\Omega$ of all real-valued, $L$-periodic, $C^{\infty}(\Omega)$-functions on $\mathbb{R}^2$. For each integer $q\geq0$, let $H^q_{per}(\Omega)$ be the closure of $C^{\infty}_{per}(\Omega)$ in the usual Sobolev norm $\|\cdot\|_q$, and $H^{-q}_{per} (\Omega)$ be the dual space of $H^{q}_{per}(\Omega)$. Note that $H^0_{per}(\Omega) = L^2(\Omega)$. 
In turn, by introducing $\mu_{1}=\frac{\delta F}{\delta \phi_{1}}=\delta_{\phi_{1}} F$ and $\mu_{2}=\frac{\delta F}{\delta \phi_{2}}=\delta_{\phi_{2}} F$, the mixed weak formulation of MMC-TDGL equations \eqref{MMC-equations} becomes: find $\phi_1,\mu_1,\phi_2,\mu_2\in L^2(0,T;H_{per}^1(\Omega))$, with $\partial_t\phi_1$, $ \partial_t\phi_2\in L^2(0,T;H_{per}^{-1}(\Omega))$, satisfying
\begin{equation}\label{weak-formula}
	\left\{
	\begin{aligned}
	{\left(\partial_t\phi_1, v_1\right)}&{+(D_1\nabla \mu_1, \nabla v_1)=0,} & {\forall v_1 \in H^{1}_{per}(\Omega)}, \\
	{(\mu_1, w_1)}&{=\left(\delta_{\phi_{1}} F\left(\phi_{1}, \phi_{2}\right), w_{1}\right),} & {\forall w_1 \in H^{1}_{per}(\Omega)},\\
	{\left(\partial_t\phi_2, v_2\right)}&{+(D_2\nabla \mu_2, \nabla v_2)=0,} & {\forall v_2 \in H^{1}_{per}(\Omega)},\\
	{(\mu_2, w_2)}&{=\left(\delta_{\phi_{2}} F\left(\phi_{1}, \phi_{2}\right), w_{2}\right),} & {\forall w_2 \in H^{1}_{per}(\Omega)},
	\end{aligned}
	\right.
\end{equation}
for any $t \in [0,T]$, where $(\cdot,\cdot)$ represents the $L^2$ inner product or the duality pairing, as appropriate.

\subsection{The finite element scheme}
The following preliminary results are associated with the existence of the convex-concave decomposition of the energy functional $F$, i.e, $F(\phi_1 , \phi_2)$ admits a (not necessarily unique) splitting into purely convex and concave energies, $F=F_c-F_e$,
where $F_c=\int_{\Omega} S \left( \phi_1  , \phi_2  \right) + K \left( \phi_1  , \phi_2 \right) d \mathbf { x }$ and $F_e=-\int_\Omega H \left( \phi_1  , \phi_2  \right) d \mathbf { x }$ are convex with respect to the specific variables.

\begin{prop} \label{Proposition1}  \cite{Dong2020b}
	Define the functions
\begin{equation*}
\begin{array}{rlrc}
T_{1}(u, v)&:=\frac{v^{2}}{36 u}, \quad u \in(0, \infty), & v &\in \R ; \\
T_{2}\left(u_{1}, u_{2}, v_{1}, v_{2}\right)&:=\frac{\left(v_{1}+v_{2}\right)^{2}}{36\left(1-u_{1}-u_{2}\right)}, & u_{1}, u_{2}, v_{1}, v_{2} &\in \R ; \\
T_{3}(u, v, w)&:=\frac{w^{2}}{36(u+v)}, & u, v, w &\in \R ;\\
T_{4}(u_{1},u_{2},u_{3},v)&:=\frac{3v^{2}}{(u_{1}+u_{2}+u_{3})} & u_{1}, u_{2}, u_{3}, v &\in \R .
\end{array}
\end{equation*}
Then,
\begin{enumerate}
	\item
	$T_{1}(u, v)$ is convex in $(0,+\infty) \times \R $.
	\item
	$T_{2}\left(u_{1}, u_{2}, v_{1}, v_{2}\right)$ is convex in $\R ^{4},$ provided $u_{1}+u_{2}<1$.
	\item
	$T_{3}(u, v, w)$ is convex in $\R ^{3},$ provided $u+v>0$.
	\item
	$T_{4}(u_{1},u_{2},u_{3},v)$ is convex in $\R ^{4}$, provided $u_{1}+u_{2}+u_{3} > 0$.
	\item 
	$S\left(u_{1}, u_{2}\right)$ is convex in the Gibbs triangle  $\mathcal{G},$ defined as
	\[
	\mathcal{G}:=\left\{\left(u_{1}, u_{2}\right) \mid u_{1}, u_{2}>0, u_{1}+u_{2}<1\right\}.
	\]
	\item
	$H\left(u_{1}, u_{2}\right)$ is concave, provided that $4 \chi_{13} \chi_{23}-\left(\chi_{12}-\chi_{13}-\chi_{23}\right)^{2}>0$.
\end{enumerate}
\end{prop}

We consider a finite element method for solving \eqref{weak-formula}. Let $\mathcal{T}_h$ be a shape-regular triangulation of $\Omega$ with mesh size $h$, denote $h_e$ the diameter of each triangle $e\in {\mathcal T}_h$ and $\triangle_e$ the area of $e$. Noticing that the element is shape regular, we can assume that $\frac{h_e^2}{\triangle_e}$ is uniformly bounded by one
constant $C_{\mathcal T}$: $\frac{h_e^2}{\triangle_e}\le C_{\mathcal T}$.
 Based on the quasi-uniform triangulated mesh $\mathcal{T}_h$, the finite element space is defined as
$$
S_h:=\{v \in H_{per}^{1}(\Omega)\mid ~v~\text{is piecewise linear on each}~e\in \mathcal{ T }_h\} = \text{span}\{\chi_{j}\mid j= 1,\cdots, N_{p}\} ,
$$
where $\chi_{j}$ is the common nodal basis function which is 1 at the node $P_{j}$ and 0 at all other nodes. Define $\mathring{S}_h:=S_h \cap L^2_0(\Omega)$, with $L^2_{0}(\Omega)=\{ v \in L^2(\Omega)\mid (v,1)=0\}$, the function space with zero mean in $L^2(\Omega)$ .
\begin{defin}
The discrete energy $E: S_{ h }\times S_{ h } \rightarrow \R $ is defined as
	$$E( \phi_{1}, \phi_{2} ) = \int_{{\Omega}} S \left( \phi_1  , \phi_2  \right) + H ( \phi_1  , \phi_2 ) + K \left( \phi_1  , \phi_2 \right) d \mathbf {x}.$$
\end{defin}
\begin{lem}\label{lem2}
	Suppose that $\Omega=(0,L)^2$ and $\phi_1, \phi_2 \in S_{ h }$ are periodic. Define the discrete energies as follows
	\begin{equation}
	\begin{aligned}
	E_c=\int_{\Omega} S \left( \phi_1  , \phi_2  \right) + K \left( \phi_1  , \phi_2 \right) d \mathbf {x},\quad
	E_e=-\int_{\Omega} H \left( \phi_1  , \phi_2  \right) d \mathbf {x} ,
	\end{aligned}
	\end{equation}
	where $S \left( \cdot  , \cdot  \right), H \left( \cdot  , \cdot  \right)$ and $K \left( \cdot , \cdot  \right)$ are defined by \eqref{SHK}. Then, both $E_c(\phi_1,\phi_2)$ and $E_e(\phi_1, \phi_2)$ are convex. 
\end{lem}
\begin{proof}
Since $S_{ h } \subset H_{per}^1(\Omega),$ the proof follows the analysis in Proposition~\ref{Proposition1}, and the conclusions are obvious.
\end{proof}

Then we introduce the fully-discrete scheme. Let $M$ be a positive integer and $0=t_0<t_1<\cdots<t_M= M\tau=T$ be a uniform partition of $[0,T]$, with $\tau=t_i-t_{i-1}$ and $i=1,\cdots, M.$ Due to the convex-concave decomposition $E=E_c - E_e$, the potentials could also be split into two parts, namely ${ \mu_{1} }$ and $ { \mu_{2} }$. By treating the convex term implicitly and the concave part explicitly, the first-order in time, mixed finite element scheme could be formulated as follows:
for any $0\leq n\leq M-1$, given $\phi_{1h}^n,\phi_{2h}^n \in S_{h}$, find $\phi_{1h}^{n+1}, \mu_{1h}^{n+1},\phi_{2h}^{n+1},\mu_{2h}^{n+1}\in S_{h}$ such that
\begin{subequations}\label{full-scheme}
	\begin{equation}\label{full-scheme1}
	\begin{aligned}
	\left(\frac{\phi_{1h}^{n+1}-\phi_{1h}^{n}}{\tau}, v_{1}\right)
	&=-\left(D_{1} \nabla {\mu}_{1h}^{n+1}, \nabla v_{1}\right),
	& \forall v_{1}\in \mathring{S}_h,\\
	\left(\mu_{1h}^{n+1}, w_{1}\right)
	&=\left(\delta_{\phi_{1}} E_{c}\left(\phi_{1h}^{n+1}, \phi_{2h}^{n+1}\right), w_{1}\right)
	+\left(\frac{\partial H}{\partial \phi_{1}}\left(\phi_{1h}^{n}, \phi_{2h}^{n}\right),w_{1} \right),
	& \forall w_{1}\in \mathring{S}_h,
	\end{aligned}
	\end{equation}
	\begin{equation}\label{full-scheme2}
	\begin{aligned}
	\left(\frac{\phi_{2h}^{n+1}-\phi_{2h}^{n}}{\tau}, v_{2}\right)
	&=-\left(D_{2} \nabla \mu_{2h}^{n+1}, \nabla v_{2}\right),
	& \forall v_{2}\in \mathring{S}_h,\\
	\left(\mu_{2h}^{n+1}, w_{2}\right)
	&=\left(\delta_{\phi_{2}} E_{c}\left(\phi_{1h}^{n+1}, \phi_{2h}^{n+1}\right), w_{2}\right)
	+\left(\frac{\partial H}{\partial \phi_{2}}\left(\phi_{1h}^{n}, \phi_{2h}^{n}\right),w_{2} \right),
	& \forall w_{2}\in \mathring{S}_h,
	\end{aligned}
	\end{equation}
\end{subequations}

where
\begin{equation}
\begin{aligned}
\delta_{\phi_{1}} E_{c}\left(\phi_{1h}^{n+1}, \phi_{2h}^{n+1}\right)&=\frac{\partial S}{\partial \phi_{1}}\left(\phi_{1h}^{n+1}, \phi_{2h}^{n+1}\right)+\delta_{\phi_{1}} K\left(\phi_{1h}^{n+1}, \phi_{2h}^{n+1}\right),\\
\delta_{\phi_{2}} E_{c}\left(\phi_{1h}^{n+1}, \phi_{2h}^{n+1}\right)&=\frac{\partial S}{\partial \phi_{2}}\left(\phi_{1h}^{n+1}, \phi_{2h}^{n+1}\right)+\delta_{\phi_{2}} K\left(\phi_{1h}^{n+1}, \phi_{2h}^{n+1}\right),
\end{aligned}
\end{equation}
and for $i=1,2$,
\begin{equation}
\begin{aligned}
(\delta_{\phi_{i}} K\left(\phi_{1}, \phi_{2}),w\right)=& \left(- \frac { a _ {i }^2 | \nabla \phi _ {i } |^2 } { 36 \phi _ { i }^2 }, w\right) +\left(  \frac { a _ { i}^2 \nabla  \phi _ { i } } { 18 \phi _ { i } } , \nabla w\right)\nonumber\\
&+ \left(\frac { a_3^2 | \nabla \left( 1 - \phi_1 - \phi_2 \right) |^2 } { 36 \left( 1 - \phi_1 - \phi_2 \right)^2 }, w\right)-
\left( \frac { a_3^2 \nabla \left( 1 - \phi_1 - \phi_2 \right) } { 18 \left( 1 - \phi_1 - \phi_2 \right) },\nabla w\right).
\end{aligned}
\end{equation}

\begin{defin}
	The discrete Laplacian operator $\Delta_{h}: S_{h} \rightarrow \mathring{S}_{h}$ is defined as follows: for any $v_h \in S_h, \Delta_{h}v_h \in \mathring{S}_h$ denotes the unique solution to the problem $$\left(\Delta_{h} v_{h,} \chi\right)=-\left(\nabla v_{h}, \nabla \chi\right), \quad \forall \chi \in S_{h}.$$
\end{defin}

It is straightforward to show that by restricting the domain, $\Delta_{h}: \mathring{S}_{h} \rightarrow \mathring{S}_{h}$ is invertible, and for any $v_h\in \mathring{S}_{h}$, and we have
$$
\left(\nabla\left(-\Delta_{h}\right)^{-1} v_{h}, \nabla \chi\right)= \left(v_{h}, \chi\right), \quad \forall \chi \in S_{h}.
$$

\begin{defin}
The discrete $H^{-1}$ norm $\|\cdot\|_{-1,h}$, is defined as follows:
\begin{equation}
\|v_h\|_{-1,h}:=\sqrt{(v_h,(-\Delta_{h})^{-1} v_h)}, \quad \forall v_h \in \mathring{S}_h.
\end{equation}
\end{defin}

\begin{lem}\label{lem3}
	Suppose that $\Omega = (0, L)^2$ and $\phi_1, \phi_2, \psi_{ 1 }, \psi_{ 2 } \in S_{ h }$ are periodic. Consider the convex-concave decomposition of the energy $ E(\phi_1, \phi_2)$ into $ E = E_{ c } - E_{ e }$. Then we have
	\begin{equation}
	\begin{aligned}
	E ( \phi_{1},\phi_{2}) - E(\psi_{1},\psi_{2}) &\leqslant \left(\delta_{\phi_{1}} E_{c}\left(\phi_{1}, \phi_{2}\right)-\delta_{\phi_{1}} E_{e}\left(\psi_{1},\psi_{2}\right), \phi_{1}-\psi_{1}\right)
	\\&+\left(\delta_{\phi_{2}} E_{c}\left(\phi_{1}, \phi_{2}\right)-\delta_{\phi_{2}} E_{e}\left(\psi_{1}, \psi_{2}\right),\phi_{2}-\psi_{2}\right),
	\end{aligned}
	\end{equation}
	where $\delta _ { \phi_1 } $ and $\delta _ { \phi_2 }$ denote the variational derivatives.
\end{lem}

\begin{proof}
	Define \begin{flalign*}
	e_{ c }(\mathbf{ u }, \mathbf{ p }) & = S(u, p) + T_{1}(u,u_{x})+T_{1}(u,u_{y})+T_{1}(p,p_{x})+T_{1}(p,p_{y})+T_{2}(u,p,u_{x},p_{x})+T_{2}(u,p,u_{y},p_{y}),\\
	e_{ e }(\mathbf{ u }, \mathbf{ p }) & = - H(u, p),
	\end{flalign*}
	where $\textbf{u}=(u, u_{ x }, u_{ y }), \textbf{p}= (p, p_{ x }, p_{ y })$. The following identities are obvious
	\begin{flalign*}
	E_{ c } = \int_{ \Omega } e_{ c }(\mathbf{ u }, \mathbf{ p }) d \mathbf{ x} ,  \quad 	
	E_{ e } = \int_{ \Omega } e_{ e }(\mathbf{ u }, \mathbf{ p }) d \mathbf{ x} . 	
	\end{flalign*}
	We know that both $e_{ c }(\mathbf{ u }, \mathbf{ p })$ and $e_{ e }(\mathbf{ u }, \mathbf{ p })$ are convex on $((0, 1) \times \R  \times \R )^{2}$.
	Then we have
	$$
	e_{ c }(\mathbf{ v }, \mathbf{ p }) - e_{ c }( \mathbf{ u }, \mathbf{ p }) \geq \nabla_{\mathbf{ u }}e_{c}(\mathbf{ u }, \mathbf{ p })\cdot(\mathbf{ v } - \mathbf{ u }) .
	$$
Next, setting $\mathbf{ u } = (\phi_1, \phi_{1x}, \phi_{1y}), \mathbf{ v } = (\psi_{ 1 }, \psi_{1x}, \psi_{1y}), \mathbf{ p } = (\phi_2, \phi_{2x}, \phi_{2y})$, one obtains
	\begin{flalign*}
	E_{ c }(\psi_{1}, \phi_{2}) - E_{ c }(\phi_{1}, \phi_{2}) \geq& \int_{ \Omega }\pa_{ \phi_{1} } e_{ c }(\mathbf{ u }, \mathbf{ p }) (\psi_{ 1 } - \phi_1)
	+\pa_{ \phi_{1x} } e_{ c }(\mathbf{ u }, \mathbf{ p })  (\psi_{ 1x } - \phi_{ 1x })\\
	&+\pa_{ \phi_{1y} } e_{ c }(\mathbf{ u }, \mathbf{ p })  (\psi_{ 1y } - \phi_{ 1y }) d \mathbf{ x}\\
	 =& (\delta_{ \phi_1} E_{ c }(\phi_1, \phi_2 ), \psi_{1}-\phi_{1}).
	\end{flalign*}
	Similarly, the following inequality could be derived for $E_{ e }$:
	\begin{flalign*}
	E_{ e }(\phi_1, \phi_2) - E_{e}(\psi_{1}, \phi_2) \geq (\delta_{ \phi_1} E_{ e }(\psi_{ 1 }, \phi_2), \phi_{1}-\psi_{1}).
	\end{flalign*}
Then the following estimate holds
	\begin{flalign*}
	E(\psi_{ 1 }, \phi_2) - E(\phi_1, \phi_2)
	& = E_{  c }(\psi_{ 1 }, \phi_2) - E_{ e }(\psi_{ 1 }, \phi_2)
	- (E_{ c }(\phi_1, \phi_2) - E_{ e }(\phi_1, \phi_2))\\
	& = E_{ c }(\psi_{ 1 }, \phi_2) - E_{ c }(\phi_1, \phi_2) - (E_{ e }(\psi_{ 1 }, \phi_2) - E_{ e }(\phi_1, \phi_2))\\
	& \geq (\delta_{ \phi_1} E_{ c }(\phi_1, \phi_2), \psi_{ 1 } - \phi_1) + (\delta_{ \phi_1} E_{ e }(\psi_{ 1 }, \phi_2), \phi_1 - \psi_{ 1 })\\
	&\geq (\delta_{ \phi_1} E_{ c }(\phi_1, \phi_2) - \delta_{ \phi_1} E_{ e }(\psi_{ 1 }, \phi_2), \psi_{ 1 } - \phi_1),
	\end{flalign*}
and we get
	\begin{flalign*}
	E(\phi_1, \phi_2) - E(\psi_{ 1 }, \phi_2)
	\leq (\delta_{ \phi_1} E_{ c }(\phi_1, \phi_2) - \delta_{ \phi_1} E_{ e }(\psi_{ 1 }, \phi_2), \phi_1 - \psi_{ 1 }).
	\end{flalign*}
A similar inequality could be derived in the same fashion:
	\begin{flalign*}
	E(\psi_{ 1 }, \phi_2) - E(\psi_{ 1 }, \psi_{ 2 })
	\leq (\delta_{ \phi_2} E_{ c }(\psi_{ 1 }, \phi_2) - \delta_{ \phi_2} E_{ e }(\psi_{ 1 }, \psi_{ 2 }), \phi_2 - \psi_{ 2 }).
	\end{flalign*}
	To sum up, the proof is completed.
\end{proof}

\subsection{The mass lumped finite element method} 

The standard mixed FEM~\eqref{full-scheme} leads to a theoretical difficulty for justifying the positivity-preserving property. To overcome this subtle difficulty, we apply a mass lumped FEM instead, which is a modification of standard conforming FEM for solving parabolic equations. It simplifies the computation for the inverse of a mass matrix and overcomes the shortcoming of the standard FEM that can not preserve the maximum principle for homogeneous parabolic equations. In this subsection, we extend the lumped mass FEM to solve MMC-TDGL equations.

Let $P_{e,k} (k = i, j, m) $ be the vertices of triangle $e$, and $\triangle_{e}$ be the area of triangle $e$. The generation of the lumped mass matrix can be regarded as introducing the following quadrature formula:
\begin{equation}
Q_{ h }( f )=\sum_{e \in \mathcal{ T }_{h}}Q_{e}(f),
\end{equation}
where
$$
Q_{e}(f)=\frac{\triangle_{e}}{3} \sum_{k=i,j,m} f\left( P_{e, k} \right) \approx \int_{e} f d\mathbf{x } .
$$
By the above quadrature formula, it is easy to derive $Q_{ h }(\chi_{j}, \chi_{k})=0$ for $k \neq j$, so that
\begin{equation}
\sum_{k=1}^{N_{p}}\left(\chi_{j}, \chi_{k}\right)=Q_{ h }(\chi_{j} ^ 2).
\end{equation}
Notice that $\chi_{j}\chi_{k}$ is a second-degree polynomial, thus it holds that $(\chi_{j}, \chi_{k})_{e} = \frac{1}{12}\triangle_{e}$ for $k\neq j$, and $(\chi_{j}, \chi_{j})_{e} = \frac{1}{6} \triangle_{e}$. Then we get
\begin{equation}
\sum_{k=1}^{N_{p}}\left(\chi_{j}, \chi_{k}\right)=\frac{1}{3} \operatorname{area}\left(D_{j}\right),
\end{equation}
where $D_{j}$ is the union of triangles with a vertex $P_{j}$. It is obvious that
\begin{equation}
Q_{ h }(\chi_{j} ^ 2 )= \sum_{e \in \mathcal{T}_{h}} Q_{e}\left(\chi_{j}^{2}\right)=\frac{1}{3} \operatorname{area}\left(D_{j}\right) .
\end{equation}

We may then define an approximation of the inner product in $S_{h}$ by
\begin{flalign}
	(\psi, \eta)_{Q}= Q_{h}(\psi \eta),
\end{flalign}
thus $\|\eta\|_{Q}=\sqrt{(\eta,\eta)_{Q}}$ can be denoted as a norm for any $\eta \in S_h$ and is equivalent to the standard $\|\cdot\|_{L^2}$ norm by considering each triangle separately.

To facilitate the analysis below, we have to modify the definition of the discrete Laplacian operator and the discrete $H^{-1}$ norm. In fact, the primary difference is in the integral definition.

\begin{defin}
	The discrete Laplacian operator $\Delta_{h}: S_{h} \rightarrow \mathring{S}_{h}$ is defined as follows: for any $v_h \in S_h, \Delta_{h}v_h \in \mathring{S}_h$ denote the unique solution to the problem $$\left(\Delta_{h} v_{h,} \chi\right)_{Q}=-\left(\nabla v_{h}, \nabla \chi\right), \quad \forall \chi \in S_{h}.$$
\end{defin}

It is straightforward to show that by restricting the domain, $\Delta_{h}: \mathring{S}_{h} \rightarrow \mathring{S}_{h}$ is invertible, and for any $v_h\in \mathring{S}_{h}$, and we have
$$\left(\nabla\left(-\Delta_{h}\right)^{-1} v_{h}, \nabla \chi\right)= \left(v_{h}, \chi\right)_{Q}, \quad \forall \chi \in S_{h}.$$

\begin{defin}
	The discrete $H^{-1}$ norm $\|\cdot\|_{-1,Q}$, is defined as follows:
	\begin{equation}
	\|v_h\|_{-1,Q}:=\sqrt{(v_h,(-\Delta_{h})^{-1} v_h)_{Q}}, \quad \forall v_h \in \mathring{S}_h.
	\end{equation}
\end{defin}

\begin{defin}
Define the discrete energy $\hat{E}: S_{ h }\times S_{ h } \rightarrow \R$ as follows
\begin{flalign}
\hat{E} ( \phi_{1}, \phi_{2} )
= (S \left( \phi_1  , \phi_2  \right) )_{Q}
+ (H ( \phi_1  , \phi_2 ))_{Q}
+ (\tilde{K}(\phi_1,\phi_2), 1)
\end{flalign}
where
\begin{flalign}
\tilde{K}(\phi_1,\phi_2):=\sum_{\ell=1}^{3}\frac{a_\ell^{2}}{36} \frac{|\nabla \phi_\ell|^{2}}{A(\phi_\ell)},
\end{flalign}
and the operator $A$ represents element average operator, that is,
	\[
A(\phi)|_e=\frac{1}{\triangle_{e}}\int_{e} \phi d \mathbf{ x }=\frac{1}{3}(\phi_\alpha+ \phi_\beta +\phi_\gamma).
	\]
In the last equation, $\phi_\alpha$, $\phi_\beta$, and $\phi_\gamma$, are the values of $\phi$ at the three vertices of the element $e$.
    \end{defin}
    \begin{lem}
	(Existence of a convex-concave decomposition). Suppose $(\phi_1, \phi_2)\in S_{h}$. The functions
	\begin{flalign}
	\hat{ E }_{ c }&=(S( \phi_1 , \phi_2) )_{Q}
	+ (\tilde{K}(\phi_1,\phi_2),1) , \label{Ec}\\
	\hat{ E }_{ e } &= (- H( \phi_1, \phi_2) )_{Q} , \label{Ee}
	\end{flalign}
	are convex. Therefore, $\hat{ E }( \phi_1 , \phi_2)=\hat{ E }_{ c }( \phi_1 , \phi_2) - \hat{ E }_{ e }( \phi_1 , \phi_2)$ is a convex-concave decomposition of the discrete energy.
    \end{lem}
    \begin{proof}    \ignore{
	We look at the detailed expansions of $\hat{ E }_{ c }( \phi_1 , \phi_2) $ and $ \hat{ E }_{ e }( \phi_1 , \phi_2)$:
	\begin{flalign*}
	\hat{ E }_{ c }( \phi_1 , \phi_2)& = \sum_{ e \in \mathcal{ T }_{ h } }\big( \frac{1}{3}\triangle_{e}\sum_{k=i,j,m}(S( \phi_{ 1k } , \phi_{ 2k }) )\\
	& + \frac{a_{1}^{2}}{36}\triangle_{e}(T_{4}(\phi_{ 1i },\phi_{ 1j },\phi_{ 1m }, \pa_{ x } \phi_{1})
	+ T_{4}(\phi_{ 1i },\phi_{ 1j },\phi_{ 1m }, \pa_{ y } \phi_{1}) )\\
	&  + \frac{a_{2}^{2}}{36} \triangle_{e}(T_{4}(\phi_{ 2i },\phi_{ 2j },\phi_{ 2m }, \pa_{ x } \phi_{2})
	+ T_{4}(\phi_{ 2i },\phi_{ 2j },\phi_{ 2m }, \pa_{ y } \phi_{2})
	)\\
	&  + \frac{a_{3}^{2}}{36} \triangle_{e}(T_{4}(1-\phi_{ 1i }-\phi_{2i},1-\phi_{ 1j }-\phi_{2j},1-\phi_{ 1m }-\phi_{2m}, \pa_{ x } \phi_{1}+ \pa_{ x } \phi_{2}) \\
	& + T_{4}(1-\phi_{ 1i }-\phi_{2i},1-\phi_{ 1j }-\phi_{2j},1-\phi_{ 1m }-\phi_{2m}, \pa_{ y } \phi_{1}+ \pa_{ y } \phi_{2}) ) ,
	\big)\\
	\hat{ E }_{ e }( \phi_1 , \phi_2)&=\sum_{ e \in \mathcal{ T }_{ h } }\big(\frac{1}{3}\triangle_{e}\sum_{k=i,j,m} (-H( \phi_{ 1k } , \phi_{ 2k })\big) ,
	\end{flalign*}
	where $\phi_{ lk },l=1,2,k=i,j,m$, stand for the values on the three vertices of a triangle element $e$. It is clear that $\hat{ E }_{ c }$ and $\hat{ E }_{ e }$ are linear combination of certain convex functions; see
	the analysis in Proposition \ref{Proposition1}. Therefore, they are both convex.}
The convex-concave decomposition is easily obtained by applying Proposition \ref{Proposition1}.
\end{proof}

In turn, the lumped mass form of~\eqref{full-scheme} becomes: for given $\phi_{1h}^n,\phi_{2h}^n \in S_{h}$, find $\phi_{1h}^{n+1}, \mu_{1h}^{n+1},\phi_{2h}^{n+1}$, $\mu_{2h}^{n+1}\in S_{h}$ such that
\begin{subequations}\label{Mass-lumped scheme}
	\begin{flalign}
	\left(\frac{\phi_{1h}^{n+1}-\phi_{1h}^{n}}{\tau}, v_{1}\right)_{Q}
	&=-\left(D_{1} \nabla {\mu}_{1h}^{n+1}, \nabla v_{1}\right), \label{mass1}\\
	\left({\mu_{1h}}^{n+1}, w_{1}\right)_{Q}
	&=\left(\delta_{\phi_{1}} S\left(\phi_{1h}^{n+1}, \phi_{2h}^{n+1}\right), w_{1}\right)_{Q}
	+(\delta_{\phi_{1}} \tilde{K}\left(\phi_{1h}^{n+1}, \phi_{2h}^{n+1}\right), w_{1} )\\
	&+ \left(\delta_{\phi_{1}} H\left(\phi_{1h}^{n}, \phi_{2h}^{n}\right),w_{1}\right)_{Q},\nonumber\\
	\left(\frac{\phi_{2h}^{n+1}-\phi_{2h}^{n}}{\tau}, v_{2}\right)_{Q}
	&=-\left(D_{2} \nabla \mu_{2h}^{n+1}, \nabla v_{2}\right),\label{mass2}\\
	\left({\mu_{2h}}^{n+1}, w_{2}\right)_{Q}
	&=\left(\delta_{\phi_{2}} S\left(\phi_{1h}^{n+1}, \phi_{2h}^{n+1}\right), w_{2}\right)_{Q}
	+(\delta_{\phi_{2}} \tilde{K}\left(\phi_{1h}^{n+1}, \phi_{2h}^{n+1}\right), w_{2} )\\
	&+\left(\delta_{\phi_{2}} H\left(\phi_{1h}^{n}, \phi_{2h}^{n}\right),w_{2} \right)_{Q}, \nonumber
	\end{flalign}
\end{subequations}
 where for $i=1,2$,
\begin{equation}
\begin{aligned}
(\delta_{\phi_{i}} \tilde{K}\left(\phi_{1}, \phi_{2}),w\right)=& \left(- \frac { a _ {i }^2 | \nabla \phi_{i} |^2 } { 36 (A(\phi _ { i }))^2 }, w\right) +\left(  \frac { a _ { i}^2 \nabla  \phi _ { i } } { 18 A(\phi _ { i }) } , \nabla w\right)\nonumber\\
&+ \left(\frac { a_3^2 | \nabla \left( 1 - \phi_1 - \phi_2 \right) |^2 } { 36 \left( 1 - A(\phi_1) - A(\phi_2) \right)^2 }, w\right)-
\left( \frac { a_3^2 \nabla \left( 1 - \phi_1 - \phi_2 \right) } { 18 \left( 1 - A(\phi_1) - A(\phi_2) \right) },\nabla w\right).
\end{aligned}
\end{equation}
\ignore{\begin{prop}
	Suppose $(\phi_{1}, \phi_{2})\in S_{h}$. The variational derivatives of $\hat{ E }_{ c }$ and $\hat{ E }_{ e }$ with respect to $ \phi_1$ and $\phi_2$ are grid functions satisfying
	\begin{flalign}
	\delta_{ \phi_{ i } }\hat{ E }_{ c }(\phi_{1}, \phi_{2})
	&=\frac{\pa}{\pa_{\phi_{ i }}}S(\phi_{1}, \phi_{2})
	+ \frac{\pa}{\pa_{\phi_{ i }}}\tilde{K}(\phi_{1}, \phi_{2})\\
	\delta_{ \phi_{ i } }\hat{ E }_{ e }(\phi_{1}, \phi_{2})
	&=-\frac{\pa}{\pa_{ \phi_{i} }}H(\phi_{1}, \phi_{2}), \quad \text{for i=1,2} ,
	\end{flalign}	
\end{prop}
where
\begin{flalign*}
	\frac{\pa}{\pa_{\phi_1}}\tilde{K}(\phi_{1}, \phi_{2})
	&=- \frac { a_1^2 | \nabla \phi_1 |^2 } { 36 A(\phi_1)^2 }
	-\nabla \cdot\left( \frac { a_1^2 \nabla  \phi_1 } { 18 A(\phi_1 )}\right)
	+ \frac { a_3^2 | \nabla \left( 1 - \phi_1 - \phi_2 \right) |^2 } { 36 (A\left( 1 - \phi_1 - \phi_2 \right))^2 }
	+\nabla \cdot\left( \frac { a_3^2 \nabla \left( 1 - \phi_1 - \phi_2 \right) } { 18 A\left( 1 - \phi_1 - \phi_2 \right) }\right) , \\
	\frac{\pa}{\pa_{\phi_2}}\tilde{K}(\phi_{1}, \phi_{2})
	&=- \frac { a_2^2 | \nabla \phi_2 |^2 } { 36 A(\phi_2)^2 }
	-\nabla \cdot\left( \frac { a_2^2 \nabla  \phi_2 } { 18 A(\phi_2 )}\right)
	+ \frac { a_3^2 | \nabla \left( 1 - \phi_1 - \phi_2 \right) |^2 } { 36 (A\left( 1 - \phi_1 - \phi_2 \right))^2 }
	+\nabla \cdot\left( \frac { a_3^2 \nabla \left( 1 - \phi_1 - \phi_2 \right) } { 18 A\left( 1 - \phi_1 - \phi_2 \right) }\right) .
\end{flalign*}}

In addition, the following lemma is needed for the later analysis.
\begin{lem} \label{lem4}
	Suppose that $\Omega = (0,L)^2$ and $\phi_1,\phi_2, \varphi_1,\varphi_2: \Omega \rightarrow \R $ are periodic and sufficiently regular. Consider the convex-concave decomposition of the energy $\hat{ E }(\phi_1,\phi_2)$ into $\hat{ E } = \hat{ E }_c-\hat{ E }_e$, given by \eqref{Ec}-\eqref{Ee}, then we have
	\begin{flalign}
	\hat{ E }(\phi_{1},\phi_{2})-\hat{ E }(\varphi_{1},\varphi_{2})
	&\leq \left(\frac{\pa}{\pa{\phi_1}}S(\phi_{1}, \phi_{2})+\frac{\pa}{\pa{\phi_1}}H(\phi_{1}, \phi_{2}),\varphi_{ 1 } - \phi_1\right)_{Q}
	+( {\delta_{\phi_1}} \tilde{K}(\phi_{1}, \phi_{2}),\varphi_{ 1 } - \phi_1)\nonumber\\
	&+\left(\frac{\pa}{\pa{\phi_2}}S(\phi_{1}, \phi_{2}) +\frac{\pa}{\pa{\phi_2}}H(\phi_{1}, \phi_{2}),\varphi_{ 2 } - \phi_2\right)_{Q}
	+( {\delta_{\phi_2}} \tilde{K}(\phi_{1}, \phi_{2}),\varphi_{ 2 } - \phi_2) . \nonumber
	\end{flalign}
\end{lem}
\begin{proof}
	Fix $(\phi_1,\phi_2) \in S_{h} \times S_{h}$ and $(\psi_{1},\psi_{2})\in S_{h}\times S_{h}$. For any $0<\lambda<1,$ we can define the continuous and differentiable function $J_{c}(
	\lambda):=\hat{ E }_{c}(\phi_{1} + \lambda\psi_{1},\phi_{2} + \lambda\psi_{2}).$ Since $\hat{ E }_{c}(\phi_{1},\phi_{2})$ is convex, $J_{c}(\lambda)$ is convex. We have $J_{c}(\lambda)-J_{c}(0)\geq J_{c}^{\prime}(0)\lambda$. This implies that
	\begin{flalign}
	\hat{ E }_{c}(\phi_{1} +\lambda\psi_{ 1 },\phi_{2} +\lambda\psi_{ 2 })-\hat{ E }_{c}(\phi_{1},\phi_{2})
	&\geq \left(\frac{\pa}{\pa{\phi_1}}S(\phi_{1}, \phi_{2}),\lambda\psi_{ 1 }\right)_{Q}
	+\left(\frac{\pa}{\pa{\phi_2}}S(\phi_{1}, \phi_{2}),\lambda\psi_{ 2 }\right)_{Q}\\
	&- \left(\frac { a_1^2 | \nabla \phi_1 |^2 } { 36 (A(\phi_1))^2 } ,\lambda \psi_{ 1 }\right)
	+ \left(  \frac { a_1^2 \nabla  \phi _ {1 } } { 18 A(\phi_1 )}, \lambda \nabla \psi_{1}\right)\nonumber\\
	&+ \left(\frac { a_3^2 | \nabla \left( 1 - \phi_1 -     \phi_2 \right) |^2 } { 36 (A\left( 1 - \phi_1 - \phi_2 \right))^2 }, \lambda\psi_{ 1 }\right)
	-\left(  \frac { a_3^2 \nabla \left( 1 - \phi_1 - \phi_2 \right) } { 18 A\left( 1 - \phi_1 - \phi_2 \right) }, \lambda\nabla \psi_{ 1 }\right) \nonumber\\
	&- \left(\frac { a_2^2 | \nabla \phi_2 |^2 } { 36 (A(\phi_2))^2 } ,\lambda \psi_{ 2 }\right)
	+\left(  \frac { a_2^2 \nabla  \phi_2 } { 18 A(\phi_2 )}, \lambda \nabla\psi_{2}\right)\nonumber\\
	&+ \left(\frac { a_3^2 | \nabla \left( 1 - \phi_1 -     \phi_2 \right) |^2 } { 36 (A\left( 1 - \phi_1 - \phi_2 \right))^2 }, \lambda\psi_{ 2 }\right)
	-\left(  \frac { a_3^2 \nabla \left( 1 - \phi_1 - \phi_2 \right) } { 18 A\left( 1 - \phi_1 - \phi_2 \right) } , \lambda\nabla \psi_{ 2 }\right) . \nonumber
	\end{flalign}
	We may assume that $(\varphi_{ 1 },\varphi_{ 2 }):=(\phi_1,\phi_2) + \lambda (\psi_{ 1 }, \psi_{ 2 }) \in S_{h} \times S_{h}$, since $\lambda$ is small in magnitude. Then we have
	\begin{flalign}
	\hat{ E }_{c}(\varphi_{1},\varphi_{2})-\hat{ E }_{c}(\phi_{1},\phi_{2})
	& \geq \left(\frac{\pa}{\pa{\phi_1}}S(\phi_{1}, \phi_{2}),\varphi_{ 1 } - \phi_1\right)_{Q}
	+\left(\frac{\pa}{\pa{\phi_2}}S(\phi_{1}, \phi_{2}),\varphi_{ 2 } - \phi_2\right)_{Q} \\
	&- \left(\frac { a_1^2 | \nabla \phi_1 |^2 } { 36 (A(\phi_1))^2 } ,\varphi_{ 1 } - \phi_1\right)
	+ \left(  \frac { a_1^2 \nabla  \phi _ {1 } } { 18 A(\phi_1 )}, \nabla( \varphi_{ 1 } - \phi_1)\right)\nonumber\\
	&+ \left(\frac { a_3^2 | \nabla \left( 1 - \phi_1 -     \phi_2 \right) |^2 } { 36 (A\left( 1 - \phi_1 - \phi_2 \right))^2 }, \varphi_{ 1 } - \phi_1\right)
	-\left( \frac { a_3^2 \nabla \left( 1 - \phi_1 - \phi_2 \right) } { 18 A\left( 1 - \phi_1 - \phi_2 \right) } , \nabla(\varphi_{ 1 } - \phi_1)\right) \nonumber\\
	&- \left(\frac { a_2^2 | \nabla \phi_2 |^2 } { 36 (A(\phi_2))^2 } ,\varphi_{ 2 } - \phi_2\right)
	+ \left(  \frac { a_2^2 \nabla  \phi_2 } { 18 A(\phi_2 )}, \nabla(\varphi_{ 2 } - \phi_2)\right)\nonumber\\
	&+ \left(\frac { a_3^2 | \nabla \left( 1 - \phi_1 -     \phi_2 \right) |^2 } { 36 (A\left( 1 - \phi_1 - \phi_2 \right))^2 }, \varphi_{ 2 } - \phi_2\right)
	- \left(  \frac { a_3^2 \nabla \left( 1 - \phi_1 - \phi_2 \right) } { 18 A\left( 1 - \phi_1 - \phi_2 \right) } , \nabla(\varphi_{ 2 } - \phi_2)\right) . \nonumber
	\end{flalign}
	For $\hat{ E }_{e}$, a similar inequality is available:
	\begin{flalign}
	\hat{ E }_{e}(\varphi_{1},\varphi_{2})-\hat{ E }_{e}(\phi_{1},\phi_{2})
	\geq \left(-\frac{\pa}{\pa{\phi_1}}H(\phi_{1}, \phi_{2}),\varphi_{ 1 } - \phi_1\right)_{Q}
	+ \left(-\frac{\pa}{\pa{\phi_2}}H(\phi_{1}, \phi_{2}),\varphi_{ 2 } - \phi_2\right)_{Q} .
	\end{flalign}
	Combining the inequalities, we have
	\begin{flalign}
	\hat{ E }(\phi_{1},\phi_{2})-\hat{ E }(\varphi_{1},\varphi_{2})
	&= (\hat{ E }_{c}(\phi_{1},\phi_{2})-\hat{ E }_{c}(\varphi_{1},\varphi_{2}))
	-(\hat{ E }_{e}(\phi_{1},\phi_{2})-\hat{ E }_{e}(\varphi_{1},\varphi_{2}))\\
	&\leq \left(\frac{\pa}{\pa{\phi_1}}S(\phi_{1}, \phi_{2}), \phi_1 -\varphi_{ 1 } \right)_{Q}
	+( {\delta_{\phi_1}} \tilde{K} (\phi_{1}, \phi_{2}) ,\phi_1 -\varphi_{ 1 } ) \nonumber\\
	&+\left(\frac{\pa}{\pa{\phi_2}}S(\phi_{1}, \phi_{2}),\phi_2 -\varphi_{ 2 }\right)_{Q}
	+( {\delta_{\phi_2}} \tilde{K} (\phi_{1}, \phi_{2}) ,\phi_2 -\varphi_{ 2 } ) \nonumber\\
	&-\left(-\frac{\pa}{\pa{\phi_1}}H(\phi_{1}, \phi_{2}),\phi_1-\varphi_{ 1 }\right)_{Q}
	- \left(-\frac{\pa}{\pa{\phi_2}}H(\phi_{1}, \phi_{2}),\phi_2-\varphi_{ 2 }\right)_{Q} .
	\nonumber
	\end{flalign}
Consequently, the proof is completed.
\end{proof}

\begin{rem} 
The periodic boundary condition is considered in this article, for simplicity of presentation, since all the boundary integral terms will cancel, so that the integration by parts is always valid. As a result, for all the nonlinear and singular terms, the boundary terms will cause any scientific difficulty in the mathematical analysis. Meanwhile, the analysis in this work could be extended to other type of physically relevant boundary condition, such a homogeneous Neumann one. In fact, a natural boundary condition (corresponding to the Neumann boundary one) is more straightforward in the finite element set-up, and this extension analysis will be considered in the future works.  
\end{rem} 

\section{The unique solvability and positivity-preserving property}  \label{sec:positivity}

The mass lumped method, improving the original mass matrix,  provides us with an efficient way to derive the theoretical proof of preserving positivity property for the MMC-TDGL equations.

\begin{lem}\label{lem5}\cite{chen19a}
	Suppose that $\xi,\bar{\xi}\in S_{h}$, with $(\xi-\bar{\xi},1)=0$, that is, $\xi-\bar{\xi}\in \mathring{S}_h$, and assume that $\|\xi\|_{\infty}< 1, \|\bar{\xi}\|_{\infty} \leq M $. Then, we have the following estimate:
	\begin{equation}
	\left\|-\Delta_{h}^{-1}\left(\xi-\bar{\xi}\right)\right\|_{\infty} \leq C_{1},
	\end{equation}
	where $C_{1}> 0$ depends only upon $M$ and $\Omega$. In particular, $C_{1}$ is independent of the mesh spacing $h$ .
\end{lem}

\begin{lem}\label{lem:l1}
Let $\phi,\psi\in S_h$ and $A(\psi)>0$, then
\begin{flalign}
\left(-\frac{|\nabla \phi|^2}{36(A(\phi))^2},\psi\right)+\left(\frac{\nabla \phi}{18A(\phi)},\nabla \psi\right)&\le \frac{1}{36}\left(\frac{\nabla \psi}{A(\psi)},\nabla \psi\right),\\
\left(\frac{|\nabla \phi|^2}{36(A(\phi))^2},\psi\right)-\left(\frac{\nabla \phi}{18A(\phi)},\nabla \psi\right)&\le \left(\frac{|\nabla \phi|^2}{18(A(\phi))^2},A(\psi)\right)+\frac{1}{36}\left(\frac{\nabla \psi}{A(\psi)},\nabla \psi\right).
\end{flalign}
\end{lem}
\begin{proof} By Cauchy-Schwarz inequality, on every element $e\in {\mathcal{T}}_h$, one gets
\begin{flalign}
\left|\left(\frac{\nabla \phi}{18A(\phi)},\nabla \psi\right)_e\right|
&\le \left(\frac{|\nabla \phi|^2}{36(A(\phi))^2},A (\psi) \right)_e^{\frac{1}{2}} \left(\frac{\nabla\psi}{9A(\psi)},\nabla \psi\right)_e^{\frac{1}{2}} \nonumber
\\
&\le \left(\frac{|\nabla \phi|^2}{36(A(\phi))^2},A(\psi)\right)_e  +\left(\frac{\nabla\psi}{36A(\psi)},\nabla \psi\right)_e.
\end{flalign}
Summing over all $e$,
\begin{flalign}\label{lem:l1:i1}
\left|\left(\frac{\nabla \phi}{18A(\phi)},\nabla \psi\right)\right|
\le \left(\frac{|\nabla \phi|^2}{36(A(\phi))^2},A(\psi)\right)  +\left(\frac{\nabla\psi}{36A(\psi)},\nabla \psi\right).
\end{flalign}
Moreover, note that for $\phi,\psi\in S_h$,
$$
\left(-\frac{|\nabla \phi|^2}{36(A(\phi))^2},\psi\right) =\left(-\frac{|\nabla \phi|^2}{36(A(\phi))^2},A(\psi)\right),
$$
then the lemma can be proved by combining this relationship with \eqref{lem:l1:i1}.
\end{proof}

\begin{rem}\label{rmk:r1}
If $A(\psi)\ge 0$, then Lemma~\ref{lem:l1} will be modified by 
\begin{flalign}
\left(-\frac{|\nabla \phi|^2}{36(A(\phi))^2},\psi\right)+\left(\frac{\nabla \phi}{18A(\phi)},\nabla \psi\right)&\le \frac{1}{36}\sum_{\stackrel{e\in {\mathcal T}_h}{A(\psi)>0} }\left(\frac{\nabla \psi}{A(\psi)},\nabla \psi\right)_e,\\
\left(\frac{|\nabla \phi|^2}{36(A(\phi))^2},\psi\right)-\left(\frac{\nabla \phi}{18A(\phi)},\nabla \psi\right)&\le \left(\frac{|\nabla \phi|^2}{18(A(\phi))^2},A(\psi)\right)+\frac{1}{36}\sum_{\stackrel{e\in {\mathcal T}_h}{A(\psi)>0} } \left(\frac{\nabla \psi}{A(\psi)},\nabla \psi\right)_{e}.
\end{flalign}
\end{rem}

\begin{lem}\label{lem:l2}
For any $\phi\in S_h$, if $A(\phi)>0$ on one element $e\in {\mathcal T}_h$ with mesh size $h_e$, then we have
\begin{equation}
\frac{|\nabla \phi|}{A(\phi)}\le \frac{3\sqrt{2}h_e}{2\triangle_e}
\end{equation}
on the element $e$.
\end{lem}
\begin{proof} Let $P_i=(x_i,y_i)$ $(i=1,2,3)$ be the three
vertex points of $e$, then
\begin{flalign*}
 \frac{\partial \phi}{\partial x} & =\frac{1}{2\triangle_e}( \phi(P_1)(y_2-y_3)+\phi(P_2)(y_3-y_1)+\phi(P_3)(y_1-y_2)),\\
\frac{\partial \phi}{\partial y} &=\frac{1}{2\triangle_e}( \phi(P_1)(x_3-x_2)+\phi(P_2)(x_1-x_3)+\phi(P_3)(x_2-x_1)).
\end{flalign*}
So $\nabla \phi$ can be bounded by
$$
|\nabla \phi|\le \frac{\sqrt{2}h_e}{2\triangle_e}( \phi(P_1)+ \phi(P_2) +\phi(P_3)).
$$
Note that
$$
  A(\phi)=\frac{1}{3}\left(\phi(P_1)+ \phi(P_2) +\phi(P_3)\right),
$$
then the lemma is proved.
\end{proof}

Subsequently, a combination of Lemma ~\ref{lem:l1} and Lemma ~\ref{lem:l2} leads to the following result.

\begin{lem}\label{lem:l3}
Let $\phi,\psi\in S_h$ and $A(\psi)\ge 0$,  then
\begin{flalign}
\left(-\frac{|\nabla \phi|^2}{36(A(\phi))^2},\psi\right)+\left(\frac{\nabla \phi}{18A(\phi)},\nabla \psi\right)&\le \frac{C_{\mathcal T}}{8}
\sum_{e\in {\mathcal T}_h}A(\psi)|_e,\\
\left(\frac{|\nabla \phi|^2}{36(A(\phi))^2},\psi\right)-\left(\frac{\nabla \phi}{18A(\phi)},\nabla \psi\right)&\le \frac{3C_{\mathcal T}}{8}
\sum_{e\in {\mathcal T}_h}A(\psi)|_e.
\end{flalign}
\end{lem}
\begin{proof} By Lemma ~\ref{lem:l2} and noticing that $A(\psi)$ is constant on every element $e$, we get
\begin{flalign*}
  \frac{1}{36}\sum_{\stackrel{e\in {\mathcal T}_h}{A(\psi)>0} } \left(\frac{\nabla \psi}{A(\psi)},\nabla \psi\right) & = \frac{1}{36}\sum_{\stackrel{e\in {\mathcal T}_h}{A(\psi)>0} } \left(\frac{|\nabla \psi|^2}{(A(\psi))^2},A(\psi)\right)
  \le \frac{1}{36}\sum_{e\in {\mathcal T}_h}\frac{9h_e^2}{2\triangle_e^2}(A(\psi),1)_e
  = \frac{1}{8}\sum_{e\in {\mathcal T}_h}\frac{h_e^2}{\triangle_e}A(\psi)|_e.
\end{flalign*}
Similarly,
$$
\left(\frac{|\nabla \phi|^2}{18(A(\phi))^2},A(\psi)\right) \le \frac{1}{4}\sum_{e\in {\mathcal T}_h}\frac{h_e^2}{\triangle_e}A(\psi)|_e.
$$
Since the element is shape regular, $\frac{h_e^2}{\triangle_e}\le C_{\mathcal T}$, now the lemma is proved by using Lemma ~\ref{lem:l1} and Remark \ref{rmk:r1}.
\end{proof}

\begin{theorem} \label{thm:positivity}
Given $\phi_{1}^{n},\phi_{2}^{n}\in S_{h}$, with $0<\phi_{1}^{n},\phi_{2}^{n}<1$, $0<\phi_{1}^{n}+\phi_{2}^{n}<1$, (so that $0<\overline{\phi_{1}^{n}},\overline{\phi_{2}^{n}}<1$), there exists a unique solution $\phi_{1}^{n+1},\phi_{2}^{n+1}\in S_{h}$ to \eqref{Mass-lumped scheme}, with $\overline{\phi_{1}^{n+1}}=\overline{\phi_{1}^{n}}$, $\overline{\phi_{2}^{n+1}}=\overline{\phi_{2}^{n}}$, $0<\phi_{1}^{n+1}, \phi_{2}^{n+1}<1$, and $0<\phi_{1}^{n+1}+\phi_{2}^{n+1}<1$.
\end{theorem}
\begin{proof}
	It is observed that, the numerical solution of \eqref{Mass-lumped scheme} is a minimizer of the following discrete energy functional with respect to $\phi_{1}$ and $\phi_{2}$
	\begin{equation}
	\begin{aligned}
	\mathcal{J}_{h}^{n}\left(\phi_{1}, \phi_{2}\right)
	&=\frac{1}{2 D_1\tau}\left\|\phi_{1}-\phi_{1}^{n}\right\|_{-1, Q}^{2}
	+\frac{1}{2 D_2\tau}\left\|\phi_{2}-\phi_{2}^{n}\right\|_{-1, Q}^{2}
	+\left( S\left(\phi_{1}, \phi_{2}\right), 1\right)_{Q} \\
	&+( \tilde{K}\left(\phi_{1}, \phi_{2}\right), 1 )
	+\left(\delta_{\phi_{1}} H\left(\phi_{1}^{n}, \phi_{2}^{n}\right), \phi_{1}\right)_{Q} 
	+\left(\delta_{\phi_{2}} H\left(\phi_{1}^{n}, \phi_{2}^{n}\right), \phi_{2}\right)_{Q} ,
	\end{aligned}
	\end{equation}
	over the admissible set
	\begin{equation}
	\begin{aligned}
	A_{h}:=\{\left(\phi_{1}, \phi_{2}\right) \in S_{h} \times S_{h} \mid \
	& \phi_{1}, \phi_{2} \geq 0, 0 \leq \phi_{1}+\phi_{2} \leq 1,\\
	& \left( \phi_1 - \bar{ \phi_1^{ 0 }}, 1 \right)_{Q} = 0, \quad
	  \left( \phi_2 - \bar{ \phi_2^{ 0 }}, 1 \right)_{Q} = 0\} \subset { \R ^{2 N_{p}^{ 2 }}} .
	\end{aligned}
	\end{equation}
	It is easy to observe that $\mathcal{ J}_{h}^{n}$ is a strictly convex functional with respect to $\phi_{1}$ and $\phi_{2}$ over
	this domain.
	Consider the following closed domain:
	\begin{equation}
	\begin{aligned}
	A_{h,\delta}:=\{\left(\phi_{1}, \phi_{2}\right) \in S_{h} \times S_{h} \mid\
	& \phi_{1}, \phi_{2} \geq g(\delta),\delta \leq \phi_{1}+\phi_{2} \leq 1-\delta,\\
	& \left( \phi_1 - \bar{ \phi_1^{ 0 }}, 1 \right)_{Q} = 0, 
      \left( \phi_2 - \bar{ \phi_2^{ 0 }}, 1 \right)_{Q}= 0 \} \subset { \R ^{2 N_{p}^{ 2 }}} .
	\end{aligned}
	\end{equation}
	Since $A_{h,\delta}$ is a bounded, compact and convex set in the following hyperplane $V$ in $\R ^{2 N_{p}^{2}}$,
	with dimension $2N_{p}^{2}-2$:
	\begin{equation}
	V=\left\{\left(\phi_{1}, \phi_{2}\right): \frac{1}{|\Omega|}( \phi_1, 1 )_{Q} = \overline{ \phi_1^{ 0 }}, \, \, \,
	\frac{1}{|\Omega|}( \phi_2, 1 )_{Q} = \overline{ \phi_2^{ 0 }}\right\},
	\end{equation}
	there exists a (may not unique) minimizer of $\mathcal{J}_{h}^{n}(\phi_{1},\phi_{2})$ over $A_{h,\delta}$. The key point of the
	positivity analysis is that, such a minimizer could not occur on the boundary points (in
	$V$) if $\delta$ and $g(\delta)$ are small enough.
	
	Assume the minimizer of $\mathcal{J}_{h}^{n}(\phi_{1},\phi_{2})$ occurs at a boundary point of $A_{h,\delta}$.
	
	Case 1: We set the minimization point as $(\phi_{1}^{\star},\phi_{2}^{\star})$, with $\phi_{1}^{\star}({P_{\alpha_{0}}}) :=\phi_{1,\alpha_{0}}^{\star}=g(\delta)$. In
	addition, we assume that $\phi_{1}^{\star}$ reaches the maximum value at $\alpha_{1}$, so it is obvious that $\phi_{1,\alpha_{1}}^{\star}\geq \overline{\phi_{1}^{\star}}=\overline{\phi_{1}^{0}}$. We can view the variable $\phi_{1,\alpha_{1}}$ as the $N_{p}^{2}$-th one
	in the hyperplane $V$, with the condition
	\begin{equation}
	\phi_{1,\alpha_{1}}=\frac{(\overline{\phi_{1}^{0}},1)_{Q} - \sum_{i\neq\alpha_{1}}^{N_{p}} (\phi_{1,i},\chi_{i})_{Q}}
	{(\chi_{\alpha_{1}},1)_{Q}}.
	\end{equation}
	In more details, we denote the following alternate function
	\begin{equation}
	\mathcal{U}_{h}^{n}\left(\left.\left(\phi_{1, i}\right)\right|_{i \neq \alpha_{1}}, \phi_{2}\right)
	:=\mathcal{J}_{h}^{n}\left(\cdot,\left(\phi_{1}\right)_{\alpha_{1}}, \phi_{2}\right)
	=\mathcal{J}_{h}^{n}\left(\cdot, \frac{(\overline{\phi_{1}^{0}},1)_{Q}-\sum_{i \neq \alpha_{1}}^{N_{p} }(\phi_{1,i},\chi_{i})_{Q}}{(\chi_{\alpha_{1}},1)_{Q}}, \phi_{2}\right).
	\end{equation}
	By a careful calculation, we obtain the following directional derivative
	\begin{equation}\label{Uh}
	\begin{aligned}
	\left.d_{s} \mathcal{U}_{h}^{n}\left(\phi_{1}^{*}+s \psi, \phi_{2}^{\star}\right)\right|_{s=0}
	=& \frac{1}{D_1\tau} \left( -\Delta_{ h }^{ -1 } \left( \phi_1^{ \star } - \phi_1^{ n }\right), \psi\right)_{Q}
	+\left( \delta_{ \phi_1} S\left(\phi_1^{ \star }, \phi_2^{ \star } \right),  \psi\right)_{Q}\\
	&+ (\delta_{\phi_{1}} \tilde{K}\left(\phi_{1}^{\star}, \phi_{2}^{\star}\right), \psi )
	+ \left(\delta_{\phi_{1}} H\left(\phi_{1}^{n}, \phi_{2}^{n}\right), \psi\right)_{Q} ,
	\quad \forall \psi\in \mathring{S}_{h} . 	
	\end{aligned}
	\end{equation}
	This time, due to $\left(\phi_{1}^{*}+s \psi, \phi_{2}^{\star}\right)\in A_{h,\delta}$, let us pick the direction
	\begin{equation}
	\psi=\delta_{\alpha_{0}}-C_{2}\delta_{\alpha_{1}}, \quad
	C_{2}=\frac{\operatorname{area}(D_{\alpha_{0}})}{\operatorname{area}(D_{\alpha_{1}})},
	\end{equation}
	where $\delta_{\alpha_{0}} $ and $ \delta_{\alpha_{1}}$ are the basis functions on $P_{\alpha_0}$ and $P_{\alpha_1}$,
 $D_{\alpha_{0}}$ and $D_{\alpha_{1}}$ are the support of  $\delta_{\alpha_{0}} $ and $ \delta_{\alpha_{1}}$,
respectively.
	
	For the first term appearing in \eqref{Uh}, an application of Lemma~\ref{lem5} gives
	\begin{flalign}\label{L-1}
	\frac{1}{D_1\tau} (-\Delta_{ h }^{ -1 }( \phi_1^{ \star} - \phi_1^{ n }), \psi)_{Q}
	& = \frac{ 1 }{ D_1 \tau } \sum_{ e \in \mathcal{ T }_{ h } }\frac{ \triangle_{ e } }{ 3 }  \sum_{ j=1 }^{ 3 } (-\Delta_{ h }^{ -1 })( \phi_1^{ \star } - \phi_1^{ n }) \psi( P_{ e, j })\\
  \nonumber
  &=\frac{ 1 }{ 3 D_1\tau }( \operatorname{ area }( D_{ \alpha_{ 0 }}) (-\Delta_{ h }^{ -1 })( \phi_1^{ \star } \!-\! \phi_1^{ n })|_{ \alpha_{ 0 }} \!-\! C_{ 2 } \operatorname{ area }( D_{ \alpha_{ 1 }}) (-\Delta_{ h }^{ -1 })( \phi_1^{ \star } \!-\! \phi_1^{ n })|_{ \alpha_{ 1 }})\\
  \nonumber
  &=\frac{ 1 }{ 3 D_1\tau } \operatorname{ area }( D_{ \alpha_{ 0 }})\left( (-\Delta_{ h }^{ -1 })( \phi_1^{ \star } - \phi_1^{ n } )|_{ \alpha_{ 0 }} - (-\Delta_{ h }^{ -1 })( \phi_1^{ \star } - \phi_1^{ n })|_{ \alpha_{ 1 }}\right)\\
  \nonumber
  &\leq \frac{ 2C_{ 1 }}{ 3D_1\tau } \operatorname{ area }( D_{ \alpha_{ 0 }}).
	\end{flalign}
	For the second term, we see that
	\begin{flalign}
   ( \delta_{ \phi_1 }S( \phi_1^{ \star }, \phi_2^{ \star }), \psi)_{Q}
	&=\left( \frac{ 1 }{ \gamma } \ln( \frac{ \alpha \phi_1^{ \star }}{ \gamma }) - \ln( 1 - \phi_1^{ \star } - \phi_2^{ \star }), \psi\right)_{Q}\\
	\nonumber &=\sum_{e\in\mathcal{T}_h}\left(\frac{1}{3}\triangle_{ e }\sum_{j=1}^{3}\left(\frac{1}{\gamma}\ln(\frac{\alpha\phi_{1}^{\star}}{\gamma})-\ln(1-\phi_{1}^{\star}-\phi_{2}^{\star})\right)\psi(P_{e,j})\right)\\
	\nonumber &=\frac{1}{3}\operatorname{area}(D_{\alpha_{0}})\left(\left(\frac{1}{\gamma}\ln(\frac{\alpha\phi_{1}^{\star}}{\gamma})-\ln(1-\phi_{1}^{\star}-\phi_{2}^{\star})\right)|_{\alpha_{0}}\right.\\
	\nonumber
	& \left.-\left(\frac{1}{\gamma}\ln(\frac{\alpha\phi_{1}^{\star}}{\gamma})-\ln(1-\phi_{1}^{\star}-\phi_{2}^{\star})\right)|_{\alpha_{1}}\right)\\
	\nonumber
	&= \frac{ 1 }{ 3 } \operatorname{ area }(D_{ \alpha_{ 0 }})\left(\ln\frac{(\phi_{1}^{\star})^{\frac{1}{\gamma}}}{1-\phi_{1}^{\star}-\phi_{2}^{\star}}|_{\alpha_{0}}-\ln\frac{(\phi_{1}^{\star})^{\frac{1}{\gamma}}}{1-\phi_{1}^{\star}-\phi_{2}^{\star}}|_{\alpha_{1}}\right)\\
	\nonumber &\leq\frac{1}{3}\operatorname{area}(D_{\alpha_{0}})\left(\ln\frac{(g(\delta))^{\frac{1}{\gamma}}}{\delta}
	-\ln\frac{(\bar{\phi_{1}^{0}})^{\frac{1}{\gamma}}}{1-\delta}\right) \\
	\nonumber &\leq\frac{1}{3}\operatorname{area}(D_{\alpha_{0}})\left(\ln\frac{(g(\delta))^{\frac{1}{\gamma}}}{\delta}
	-\ln(\bar{\phi_{1}^{0}})^{\frac{1}{\gamma}}\right).
	\end{flalign}
	For the third term, we have
	\begin{flalign}\label{EK1}
	(\delta_{\phi_{1}} \tilde{K}\left(\phi_{1}^{\star}, \phi_{2}^{\star}\right), \psi )
	&=\left(-\frac{a_{1}^{2}|\nabla {\phi_{1}^{\star}}|^{2}}{36 {(A(\phi_{1}^{\star}))}^{2}},\psi\right)
	+\left(\frac{a_{1}^{2} \nabla {\phi_{1}^{\star}}}{18 A({\phi_{1}^{\star}})},\nabla \psi\right)\\
	\nonumber
	&+\left(\frac{a_{3}^{2}\left|\nabla\left(1-{\phi_{1}^{\star}}-{\phi_{2}^{\star}}\right)\right|^{2}}{36\left(A(1-{\phi_{1}^{\star}}-{\phi_{2}^{\star}})\right)^{2}},\psi\right) -\left(\frac{a_{3}^{2}\nabla\left(1-{\phi_{1}^{\star}}-{\phi_{2}^{\star}}\right)}{18A\left(1-{\phi_{1}^{\star}}-{\phi_{2}^{\star}}\right)},\nabla\psi \right).
\end{flalign}
By Lemma ~\ref{lem:l3},
\begin{flalign}
\left(-\frac{a_1^2|\nabla {\phi_{1}^{\star}}|^{2}}{36 {(A(\phi_{1}^{\star}))}^{2}},\psi\right)
	  +\left(\frac{ a_1^2\nabla {\phi_{1}^{\star}}}{18 A({\phi_{1}^{\star}})},\nabla \psi\right)&
\le \frac{a_1^2C_{\mathcal T}}{8} \sum_{e\in {\mathcal T}_h}A(\delta_{\alpha_0})|_e
+ \frac{3a_1^2C_2C_{\mathcal T}}{8} \sum_{e\in {\mathcal T}_h}A(\delta_{\alpha_1})|_e\\
\nonumber&=\frac{a_1^2C_{\mathcal T}}{24} \sum_{e\in D_{\alpha_0}} 1
+ \frac{a_1^2C_2C_{\mathcal T}}{8} \sum_{e\in D_{\alpha_1}} 1.
\end{flalign}
Similarly,
\begin{flalign}
\left(\frac{a_{3}^{2}\left|\nabla\left(1-{\phi_{1}^{\star}}-{\phi_{2}^{\star}}\right)\right|^{2}}{36\left(A(1-{\phi_{1}^{\star}}-{\phi_{2}^{\star}})\right)^{2}},\psi\right) -\left(\frac{a_{3}^{2}\nabla\left(1-{\phi_{1}^{\star}}-{\phi_{2}^{\star}}\right)}{18A\left(1-{\phi_{1}^{\star}}-{\phi_{2}^{\star}}\right)},\nabla\psi \right)\le
\frac{a_3^2C_{\mathcal T}}{8} \sum_{e\in D_{\alpha_0}} 1+\frac{a_3^2C_2C_{\mathcal T}}{24} \sum_{e\in D_{\alpha_1}} 1.
\end{flalign}
Then the term $\delta_{\phi_{1}} \tilde{K}$ can be bounded by
\begin{flalign} \label{deltaK}
	(\delta_{\phi_{1}} \tilde{K}\left(\phi_{1}^{\star}, \phi_{2}^{\star}\right), \psi )\le \frac{(a_1^2+3a_3^2)C_{\mathcal T}}{24} \sum_{e\in D_{\alpha_0}} 1
+ \frac{(3a_1^2+a_3^2)C_2C_{\mathcal T}}{24} \sum_{e\in D_{\alpha_1}} 1.
\end{flalign}
For the numerical solution $\phi_{1}^{n}$ at the previous time step, the a-priori assumption
$0 < \phi_{1}^{n} < 1$ indicates that
$$
  -1 \leq \phi_{1}^{n}(P_{\alpha_{ 0 }}) - \phi_{1}^{n}(P_{\alpha_{ 1 }}) \leq 1.
$$
	For the last term, we have
	\begin{flalign}
	(\delta_{\phi_{1}}H\left(\phi_{1}^{n}, \phi_{2}^{n}\right), \psi)_{Q}
	&=(\chi _ { 13 } - 2 \chi_{13} \phi_{1}^{n} +(\chi_{12} - \chi_{13} - \chi_{23}) \phi_{2}^{}, \psi )_{Q}\\
	\nonumber &=\sum_{e\in \mathcal{T}_{h}}\frac{1}{3}\triangle_{ e }\left(\sum_{j=1}^{3}(\chi_{13} -2\chi_{13}\phi_{1}^{n} - (\chi_{12}-\chi_{13}-\chi_{23})\phi_{2}^{n})\psi(P_{e,j})\right)\\
	\nonumber &=\frac{1}{3}\operatorname{area}(D_{\alpha_{0}})\left(-2\chi_{13}(\phi_{1}^{n}|_{\alpha_{0}} - \phi_{1}^{n}|_{\alpha_{1}})-(\chi_{12}-\chi_{13}-\chi_{23})(\phi_{2}^{n}|_{\alpha_{0}}-\phi_{2}^{n}|_{\alpha_{1}}) \right)\\
	\nonumber
	&\leq \frac{1}{3}\operatorname{area}(D_{\alpha_{0}})(\chi_{12} + 3  \chi_{13} + \chi_{23}).
	\end{flalign}
	To sum up, the following inequality is available
	\begin{flalign}
	\left.d_{s} \mathcal{U}_{h}^{n}\left(\phi_{1}^{*}+s \psi, \phi_{2}^{\star}\right)\right|_{s=0}
	\leq \frac{1}{3}\operatorname{area}(D_{\alpha_{ 0 }})\ln\frac{(g(\delta))^{\frac{1}{\gamma}}}{\delta}
	+r_{0} ,
	\end{flalign}
 in which $\!r_{0}\!=\!\frac{\operatorname{area}(D_{\alpha_{0}})}{3}\!\left(\!\frac{2C_{1}\!}{\!D_1\!\tau\!}\!-\ln(\bar{\phi_{1}^{0}})^{\!\frac{1}{\gamma}} \!+ \!\chi_{12} \!+\! 3\chi_{13}\! + \!\chi_{23}\!\right) \!+ \!\frac{\left(\!a_1^2 \!+ 3a_3^2\!\right)C_{\!\mathcal T}}{24} \sum_{e\in \!D_{\alpha_0}} \!1 \!+\!\frac{\left(\! 3a_1^2\! + \!a_3^2 \!\right)\!C_2\!C_{\mathcal T}}{24} \sum_{e\in D_{\alpha_1}}\! 1$.
 Note that $r_{0}$ is a constant for a fixed $\tau,h$, while it becomes singular as $\tau\longrightarrow0$. For any fixed $\tau$, we could choose $g(\delta)$ sufficiently small so that
\begin{equation}\label{c1}
\frac{1}{3}\operatorname{area}(D_{\alpha_{ 0 }})\ln\frac{(g(\delta))^{\frac{1}{\gamma}}}{\delta}
+r_{0}<0,
\end{equation}
such as $ g(\delta)=\left(\delta \exp\left(-\frac{3(r_{0}+1)}{\operatorname{area}(D_{\alpha_{ 0 }})}\right)\right)^{\gamma} $.
This in turn shows that
\begin{equation*}
d_{\mathrm{s}} \mathcal{U}_{h}^{n}\left(\phi_{1}^{\star}+s \psi, \phi_{2}^{\star}\right)|_{s=0}<0, \text { for } g(\delta)~ \text{satisfying \eqref{c1}}.
\end{equation*}
This contradicts the assumption that $\mathcal{J}_{h}^{n}$ has a minimum at $(\phi_{1}^{\star},\phi_{2}^{\star})$, since the directional derivative is negative in a direction pointing into $(A_{h,\delta})^{\circ}$, the interior of $A_{h,\delta}$.
	
Case 2: Using similar arguments as in Case 1, we can also prove that, the global minimum of $\mathcal{J}_{h}^{n}$ over $A_{h,\delta}$ could not occur on the boundary section of $\phi_{2,\alpha_{0}}=g(\delta)$, for any grid node number $\alpha_{ 0 }$, if $g(\delta)$ is small enough.
	
	Case 3: We set the minimization point as $(\phi_{1}^{\star},\phi_{2}^{\star})$, with $\phi_{1,\alpha_{0}}^{\star}+\phi_{2,\alpha_{0}}^{\star}=1-\delta$, where $\alpha_{0}$ represents $\alpha_{0}$-th grid node number, and assume that $\phi_{1,\alpha_{0}}\geq\frac{1}{3}$. In addition, $(\phi_{1}+\phi_{2},1)=\overline{\phi_{1}^{0}} + \overline{\phi_{2}^{0}}$, there exists one grid point $\alpha_{1}=(i_{1}, j_{1})$, so that $\phi_{1}^{\star}+\phi_{2}^{\star}$ reaches the maximum value. It is obvious that $\phi_{1,\alpha_{1}}^{\star}+\phi_{2,\alpha_{1}}^{\star}\leq \overline{\phi_{1}^{\star}} + \overline{\phi_{2}^{\star}}= \overline{\phi_{1}^{0}} + \overline{\phi_{2}^{0}}$. Similarly, the variable $\phi_{1,\alpha_{1}}$ could be viewed as the $N_{p}^2$-th one in the hyperplane $V$, with the condition
	\begin{equation}
	\phi_{1,\alpha_{1}}=\frac{(\overline{\phi_{1}^{0}},1)_{Q} - \sum_{i\neq\alpha_{1}}^{N_{p}} (\phi_{1, i},\chi_{i})_{Q}}
	{(\chi_{\alpha_{1}},1)_{Q}} .
	\end{equation}
	In more details, the following alternate function is introduced
	\begin{equation}
	\mathcal{U}_{h}^{n}\left(\left.\left(\phi_{1, i}\right)\right|_{i \neq \alpha_{1}}, \phi_{2}\right):=\mathcal{J}_{h}^{n}\left(\cdot,\left(\phi_{1}\right)_{\alpha_{1}}, \phi_{2}\right)=\mathcal{J}_{h}^{n}\left(\cdot, \frac{(\overline{\phi_{1}^{0}},1)_{Q}-\sum_{i \neq \alpha_{1}}^{N_{p}}(\phi_{1,i},\chi_{i})_{Q}}{(\chi_{\alpha_{1}},1)_{Q}}, \phi_{2}\right).
	\end{equation}
Again, a careful calculation implies the following directional derivative
    \begin{equation}\label{Uh4}
    \begin{aligned}
    \left.d_{s} \mathcal{U}_{h}^{n}\left(\phi_{1}^{*}+s \psi, \phi_{2}^{\star}\right)\right|_{s=0}
    =& \frac{1}{D_1\tau} \left( - \Delta_{ h }^{ -1 } \left( \phi_1^{ \star } - \phi_1^{ n }\right), \psi\right)_{Q}
    + \left( \delta_{ \phi_1} S\left(\phi_1^{ \star }, \phi_2^{ \star } \right),  \psi\right)_{Q}\\
    &+ (\delta_{\phi_{1}} \tilde{K}\left(\phi_{1}^{\star}, \phi_{2}^{\star}\right), \psi) + \left(\delta_{\phi_{1}} H\left(\phi_{1}^{n}, \phi_{2}^{n}\right), \psi\right)_{Q} ,
    \quad  \forall  \psi\in \mathring{S}_{h} .
    \end{aligned}
    \end{equation}
In this case, since $(\phi_{1}^{*}+s \psi, \phi_{2}^{\star}) \in A_{h,\delta}$, we pick the direction
	\begin{equation}
	\psi=C_{2}\delta_{\alpha_{1}}-\delta_{\alpha_{0}},C_{2}=\frac{\operatorname{area}(D_{\alpha_{0}})}{\operatorname{area}(D_{\alpha_{1}})} .
	\end{equation}
For the first term in~\eqref{Uh4}, an application of Lemma~\ref{lem5} leads to
	\begin{flalign}
	\frac{1}{D_1\tau}(-\Delta_{h}^{-1}(\phi_{1}^{\star}-\phi_{1}^{n}),\psi)_{Q}
	&=\frac{1}{D_1\tau}\sum_{e\in\mathcal{T}_{h}}\frac{1}{3}\triangle_{ e }\sum_{j=1}^{3}\Delta_{h}^{-1}(\phi_{1}^{\star}-\phi_{1}^{n})\psi(P_{e,j})\\
	\nonumber &=\frac{1}{3D_1\tau} (C_{2}\operatorname{area}(D_{\alpha_{1}})(-\Delta_{ h }^{ -1 })(\phi_{1}^{\star} \!-\!\phi_{1}^{n})|_{\alpha_{1}}
	\!-\!\operatorname{area}(D_{\alpha_{0}})(-\Delta_{ h }^{ -1 })(\phi_{1}^{\star} \!-\! \phi_{1}^{n})|_{\alpha_{0}})\\
	\nonumber &=\frac{1}{3D_1\tau}\operatorname{area}(D_{\alpha_{0}})\left((-\Delta_{ h }^{ -1 })(\phi_{1}^{\star}-\phi_{1}^{n})|_{\alpha_{1}}-(-\Delta_{ h }^{ -1 })(\phi_{1}^{\star}-\phi_{1}^{n})|_{\alpha_{0}}\right)\\
	\nonumber
	&\leq\frac{2C_{1}}{3D_1\tau}\operatorname{area}(D_{\alpha_{0}}).
	\end{flalign}
For the second term, a similar inequality could be derived
	\begin{flalign}
	( \delta_{ \phi_1} S( \phi_1^{ \star }, \phi_2^{ \star }), \psi)_{Q}
	&= \left( \frac{ 1 }{ \gamma } \ln( \frac{ \alpha \phi_1^{ \star }}{ \gamma }) - \ln( 1 - \phi_1^{ \star } - \phi_2^{ \star }), \psi\right)_{Q}\\
	\nonumber
	&=\sum_{e\in\mathcal{T}_h}\left(\frac{1}{3}\triangle_{ e }\sum_{j=1}^{3}\left(\frac{1}{\gamma}\ln(\frac{\alpha\phi_{1}^{\star}}{\gamma})-\ln(1-\phi_{1}^{\star}-\phi_{2}^{\star})\right)\psi(P_{e,j})\right)\\
	\nonumber
	&=-\frac{1}{3}\operatorname{area}(D_{\alpha_{0}})\left(\left(\frac{1}{\gamma}\ln(\frac{\alpha\phi_{1}^{\star}}{\gamma})-\ln(1-\phi_{1}^{\star}-\phi_{2}^{\star})\right)|_{\alpha_{1}}\right.\\
	\nonumber
	&-\left.\left(\frac{1}{\gamma}\ln(\frac{\alpha\phi_{1}^{\star}}{\gamma})-\ln(1-\phi_{1}^{\star}-\phi_{2}^{\star})\right)|_{\alpha_{0}}\right)\\
	\nonumber
	&=-\frac{1}{3}\operatorname{area}(D_{\alpha_{0}})\left(\ln\frac{(\phi_{1}^{\star})^{\frac{1}{\gamma}}}{1-\phi_{1}^{\star}-\phi_{2}^{\star}}|_{\alpha_{1}}-\ln\frac{(\phi_{1}^{\star})^{\frac{1}{\gamma}}}{1-\phi_{1}^{\star}-\phi_{2}^{\star}}|_{\alpha_{0}}\right)\\
	\nonumber
	&\leq\frac{1}{3}\operatorname{area}(D_{\alpha_{0}})\left(\ln\frac{1}{1-\bar{\phi_{1}^{0}}-\bar{\phi_{2}^{0}}} -\ln\frac{{\frac{1}{3}}^{1/\gamma}}{\delta}\right).
	\end{flalign}
For the third term, we have the following expansion as in \eqref{deltaK}:
	\begin{flalign}
	(\delta_{\phi_{1}} \tilde{K}\left(\phi_{1}^{\star}, \phi_{2}^{\star}\right), \psi )\le \frac{(3a_1^2+a_3^2)C_{\mathcal T}}{24} \sum_{e\in D_{\alpha_0}} 1
+ \frac{(a_1^2+3a_3^2)C_2C_{\mathcal T}}{24} \sum_{e\in D_{\alpha_1}} 1.
\end{flalign}

For the last term in~\eqref{Uh4}, since the numerical solution at the previous time step is involved, the a-priori assumption $0 < \phi_{1}^{n} < 1$ indicates that
	$$
	-1 \leq \phi_{1,{\alpha}_{0}}^{n}-\phi_{1,{\alpha}_{1}}^{n} \leq 1,
	$$
	which in turn results in the following inequality
	\begin{flalign}
	(\delta_{\phi_{1}}H\left(\phi_{1}^{n}, \phi_{2}^{n}\right), \psi)_{Q}
	&=(\chi _ { 13 } - 2 \chi_{13} \phi_{1}^{n} +(\chi_{12} - \chi_{13} - \chi_{23}) \phi_{2}^{}, \psi )_{Q}\\
	\nonumber &=\sum_{e\in \mathcal{T}_{h}}\frac{1}{3}\triangle_{ e }\left(\sum_{j=1}^{3}(\chi_{13} -2\chi_{13}\phi_{1}^{n} - (\chi_{12}-\chi_{13}-\chi_{23})\phi_{2}^{n})\psi(P_{e,j})\right)\\
	\nonumber &=\frac{1}{3}\operatorname{area}(D_{\alpha_{0}})\left(-2\chi_{13}(\phi_{1}^{n}|_{\alpha_{1}} - \phi_{1}^{n}|_{\alpha_{0}})-(\chi_{12}-\chi_{13}-\chi_{23})(\phi_{2}^{n}|_{\alpha_{1}}-\phi_{2}^{n}|_{\alpha_{0}}) \right)\\
	\nonumber &\leq \frac{1}{3}\operatorname{area}(D_{\alpha_{0}})( \chi_{12} + 3\chi_{13} + \chi_{23}) .
	\end{flalign} 	
	In turn, a summation of the above estimates yields
	\begin{equation}
	\begin{aligned}
	\left.d_{s} \mathcal{U}_{h}^{n}\left(\phi_{1}^{*}+s \psi, \phi_{2}^{\star}\right)\right|_{s=0}
	\leq  \frac{1}{3}\operatorname{area}(D_{\alpha_{0}})
	\ln\delta+r_{1},
	\end{aligned}
	\end{equation}
in which $r_{1}=\frac{1}{3}\operatorname{area}(D_{\alpha_{ 0 }})(\frac{2C_{1}}{D_1\tau}+\frac{1}{\gamma}\ln3+\ln\frac{1}{1+\bar{\phi_{1}^{0}}+\bar{\phi_{2}^{0}}}+ \chi_{12} + 3 \chi_{13} + \chi_{23})+\frac{(3a_1^2+a_3^2)C_{\mathcal T}}{24} \sum_{e\in D_{\alpha_0}} 1
+ \frac{(a_1^2+3a_3^2)C_2C_{\mathcal T}}{24} \sum_{e\in D_{\alpha_1}} 1$. Again, $r_{1}$ is a constant for a fixed $\tau$ and $h$, we could choose $\delta$ sufficiently small so that
	\begin{equation}\label{c3}
	\frac{1}{3}\operatorname{area}(D_{\alpha_{0}}) \ln \delta + r_{1} < 0,
	\end{equation}
	such as $ \delta= \exp\left(-\frac{3(r_{1} + 1 )}{\operatorname{area}(D_{\alpha_{0}})}\right)$.
	This in turn demonstrates that
	\begin{equation*}
	d_{\mathrm{s}} \mathcal{U}_{h}^{n}\left(\phi_{1}^{\star}+s \psi, \phi_{2}^{\star}\right)|_{s=0}<0, \text { for } g(\delta)~ \text{satisfy \eqref{c3}} ,
	\end{equation*}
which contradicts the assumption that $\mathcal{J}_{h}^{n}$ has a minimum at $(\phi_{1}^{\star},\phi_{2}^{\star})$, since the directional derivative is negative in a direction pointing into $(A_{h,\delta})^{\circ}$, the interior of $A_{h,\delta}$.

Case 4: Using similar arguments, we can also prove that, the global minimum of $\mathcal{ J}_{h}^{n}$ over $\mathcal{ A }_{h,\delta}$ could not occur on the boundary section where $\phi_{1,\alpha_{0}}^{\star}+\phi_{2,\alpha_{0}}^{\star}=1-\delta$, if $\delta$ is sufficiently small, for any point index $\alpha_{0}$. The details are left to the interested readers.

Finally, a combination of these four cases shows that, the global minimizer of $\mathcal{J}_{h}^{n} (\phi_1,\phi_2)$ could only possibly occur at interior point of $( A_{h,\delta} )^0 \subset (A_h )^0$. We conclude that there must be a solution $(\phi_1, \phi_2) \in (A_h )^0$ that minimizes $\mathcal{J}_h^n (\phi_1,\phi_2)$ over $A_h$, which is equivalent to the numerical solution of~\eqref{Mass-lumped scheme}. The existence of the numerical solution is established.

In addition, since $\mathcal{J}_h^n(\phi_1,\phi_2)$ is a strictly convex function over $A_h$, the uniqueness analysis for this numerical solution is straightforward. The proof of Theorem~\ref{thm:positivity} is complete.
\end{proof}

\section{The energy stability} \label{sec:energy stability}

An unconditional energy stability for the proposed numerical scheme~\eqref{Mass-lumped scheme} is stated below.

\begin{theorem}(energy stability) \label{thm:energy stability}
	The unique solution of the mass lumped fully-discrete scheme \eqref{Mass-lumped scheme} is unconditionally energy stable, i.e., for any time step size $\tau > 0$, the following estimate is valid:
	\begin{equation}
	\hat{ E }( \phi_{ 1h }^{ n+1 }, \phi_{ 2h }^{ n+1 } ) \leq \hat{ E }(\phi_{ 1h }^{ n }, \phi_{ 2h }^{ n }).
	\end{equation}
\end{theorem}
\begin{proof}
	The energy stability of the mass lumped scheme \eqref{Mass-lumped scheme} is a direct consequence of Lemma~\ref{lem4}.
	
	For $w_{1},w_{2} \in \mathring{S}_h $, we denote $v_{1}=(-\Delta_h)^{-1}w_{1}, v_{2}=(-\Delta_h)^{-1}w_{2}$, and obtain
	\begin{equation}
	\begin{aligned}
	\left(\frac{ \phi_{ 1h }^{n+1}-\phi_{ 1h }^{n}}{D_1\tau}, (-\Delta_h)^{-1} w_{1}\right)_{Q}
	+\left( \delta_{\phi_1} \tilde{K} (\phi_{ 1h }^{n+1}, \phi_{ 2h }^{n+1}) ,w_{1}\right)
	&\\
	+\left(\frac{\pa}{\pa{\phi_1}}S(\phi_{ 1h }^{n+1}, \phi_{ 2h }^{n+1})+\frac{\pa}{\pa{\phi_1}}H(\phi_{ 1h }^{n+1}, \phi_{ 2h }^{n+1}),w_{1}\right)_{Q}
	& = 0,\\
	\left(\frac{ \phi_{ 2h }^{n+1}-\phi_{ 2h }^{n}}{D_2\tau}, (-\Delta_h)^{-1} w_{2}\right)_{Q}
	+\left( \delta_{\phi_2} \tilde{K} (\phi_{ 1h }^{n+1}, \phi_{ 2h }^{n+1}) ,w_{2}\right)
	&\\
	+ \left(\frac{\pa}{\pa{\phi_2}}S(\phi_{ 1h }^{n+1}, \phi_{ 2h }^{n+1})+\frac{\pa}{\pa{\phi_2}}H(\phi_{ 1h }^{n+1}, \phi_{ 2h }^{n+1}),w_{2}\right)_{Q}
	& = 0 .
	\end{aligned}
	\end{equation}
In turn, by setting $w_{1}=\phi_{ 1h }^{ n+1 } - \phi_{ 1h }^{ n }, w_{ 2 } = \phi_{ 2h }^{ n+1 } - \phi_{ 2h }^{ n }$, and applying Lemma~\ref{lem4}, we arrive at
	\begin{flalign*}
	0& =\frac{1}{D_1\tau}\|\phi_{ 1h }^{ n+1 } - \phi_{ 1h }^{ n }\|_{ -1, Q}^{ 2 }
	+ \left(\frac{\pa}{\pa{\phi_1}}S(\phi_{ 1h }^{n+1}, \phi_{ 2h }^{n+1})+\frac{\pa}{\pa{\phi_1}}H(\phi_{ 1h }^{n+1}, \phi_{ 2h }^{n+1}),\phi_{ 1h }^{ n+1 } - \phi_{ 1h }^{ n }\right)_{Q} \\
	& + \frac{1}{ D_2\tau }\|\phi_{ 2h }^{ n+1 } - \phi_{ 2h }^{ n }\|_{ -1, Q } ^{ 2 }
	+ \left(\frac{\pa}{\pa{\phi_2}}S(\phi_{ 1h }^{n+1}, \phi_{ 2h }^{n+1})+\frac{\pa}{\pa{\phi_2}}H(\phi_{ 1h }^{n+1}, \phi_{ 2h }^{n+1}),\phi_{ 2h }^{ n+1 } - \phi_{ 2h }^{ n }\right)_{Q}\\
	&+\left( \delta_{\phi_1} \tilde{K} (\phi_{ 1h }^{n+1}, \phi_{ 2h }^{n+1}) ,\phi_{ 1h }^{ n+1 } - \phi_{ 1h }^{ n }\right)
	+\left( \delta_{\phi_2} \tilde{K} (\phi_{ 1h }^{n+1}, \phi_{ 2h }^{n+1}) ,\phi_{ 2h }^{ n+1 } - \phi_{ 2h }^{ n }\right)\\
	& \geq\frac{1}{D_1\tau }\|  \phi_{ 1h }^{ n+1 } - \phi_{ 1h }^{ n }\|_{ -1, Q }^{ 2 }
	+ \frac{1}{D_2\tau }\| \phi_{ 2h }^{ n+1 } - \phi_{ 2h }^{ n }\|_{ -1, Q }^2
	+\hat{E} ( \phi_{ 1h }^{ n+1 }, \phi_{ 2h }^{ n+1 })
	- \hat{E} ( \phi_{ 1h }^{ n }, \phi_{ 2h }^{ n })\\
	&\geq \hat{E} ( \phi_{ 1h }^{ n+1 }, \phi_{ 2h }^{ n+1 })
	- \hat{E} ( \phi_{ 1h }^{ n }, \phi_{ 2h }^{ n }) .
	\end{flalign*}
This finishes the proof of Theorem~\ref{thm:energy stability}.
\end{proof}

\section{Numerical results} \label{sec:numerical results}

In this section, we perform some numerical simulations using the proposed scheme \eqref{Mass-lumped scheme}. 
In \cite{ji2021modeling}, the authors simulated several numerical examples for solving three-component MMC-TDGL equations by the SAV method and showed some phase transition processes, with different initial concentrations as well as the statistical segment lengths $a_{ i }, i=1,2,3$, consistent with an earlier work~\cite{li_phase_2015}. The statistical segment lengths $a_{ i }$ in the deGennes interfacial gradient terms, $\frac{1}{36}\sum_{i=1}^{3}\frac{a_{i}^{2}}{\phi_{ i }}|\nabla \phi_{ i }|^{ 2 }, i=1,2,3$, determine the interface thickness. Now, the default parameter of MMC-TDGL is selected to make $F_{e}$ convex; 
see Table \ref{tab_coef}. In fact, these parameters are only used for the numerical experiments, to validate the effectiveness of the proposed finite element scheme. In the numerical simulation of more realistic physical problems, these parameters could be easily adjusted, and no essential pattern difference is expected for the computational results with the parameter modification.

\begin{table}[!htpb]
	\begin{tabular}{|c|c|c|c|c|c|c|c|c|c|c|}
		\hline
		Parameter	& $D_{1}$  & $D_{2}$  & $\chi_{12}$ & $\chi_{13}$ & $\chi_{23}$  & $\gamma$ & N & $a_1$& $a_2$ & $a_3$  \\
		\hline
		Value	& 1  & 1 & 4 & 10  & 1.6 & 0.16 & 5.12 & 1 & 1 & 1\\
		\hline
	\end{tabular} \caption{The values of the parameters in the simulation}\label{tab_coef}
\end{table}

The first example is aimed to test the numerical convergence. The second one simulates a periodic structure on a large domain.  In addition, the third one is designed to show some realistic results associated with the evolution of macromolecular microsphere hydrogels. For convenience, we only consider the periodic boundary condition, and the case of homogeneous Neumann boundary condition could be similarly handled.

\begin{example}	
Let parameter $a_{1}=a_{2}=a_{3}=0.3,$ while keeping the other default parameters constant. Consider the MMC-TDGL equation over the domain $\Omega = (0,1)^2$, with the initial data given by
	\begin{equation}\label{initial-1}
	\begin{aligned}
	\phi_{1}(x, y,0)&=0.1+0.01\cos(2\pi x)\cos(2\pi y),\\
	\phi_{2}(x, y,0)&=0.5+0.01\cos(2\pi x)\cos(2\pi y).
	\end{aligned}
	\end{equation}
	\end{example}
	
We use the triangular mesh with size $h=1/256$ for partition of the domain. Since the exact solution is unknown, we compute the errors by adjacent time step in the numerical accuracy test. 
Figure~\ref{fig_conv_t} presents the $L^{\infty}$ and $L^2$ numerical errors of the three-phase variables, $\phi_{1}, \phi_{2}, \phi_{3}$, as well as a reference line at the terminal time $T=0.02$. In turn, the time step size is determined by the formula $\tau = \frac{T}{N_T}$, in which $N_T$ stands for the total number of time steps. Due to the $O (h^2)$ approximation in space, the spatial error is negligible. 
The expected temporal numerical accuracy assumption $e=C \tau$ indicates that $\ln |e|=\ln (CT) - \ln N_T$, so that we plot $\ln |e|$ versus. $\ln N_T$ to demonstrate the temporal convergence order. The reference line has an exact slope of -1, while the least square approximation to the $L^2$ error curves has approximate slopes -1.0466,  -1.0122, -1.0154, for the variables $\phi_1$, $\phi_2$ and $\phi_3$, respectively. In other words, a perfect first order temporal convergence rate is reported.  
	\begin{figure}[!htpb]
		\begin{center}
			\includegraphics[width=0.48\textwidth]{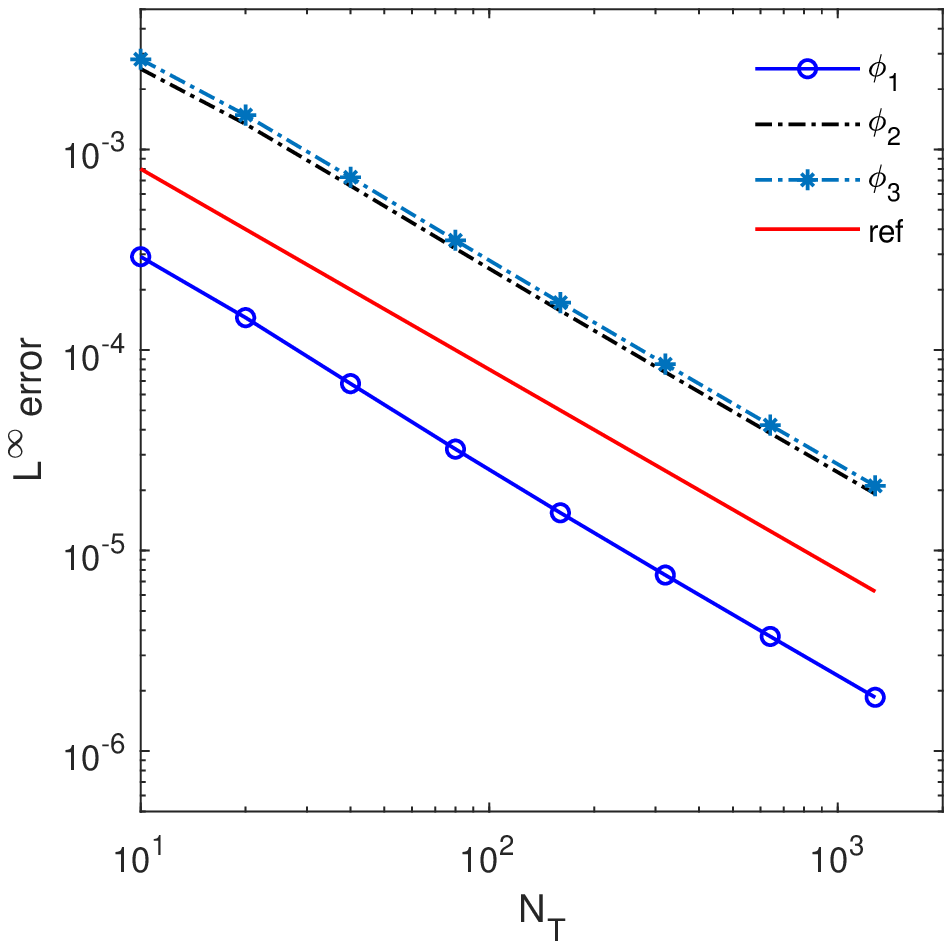}	\includegraphics[width=0.48\textwidth]{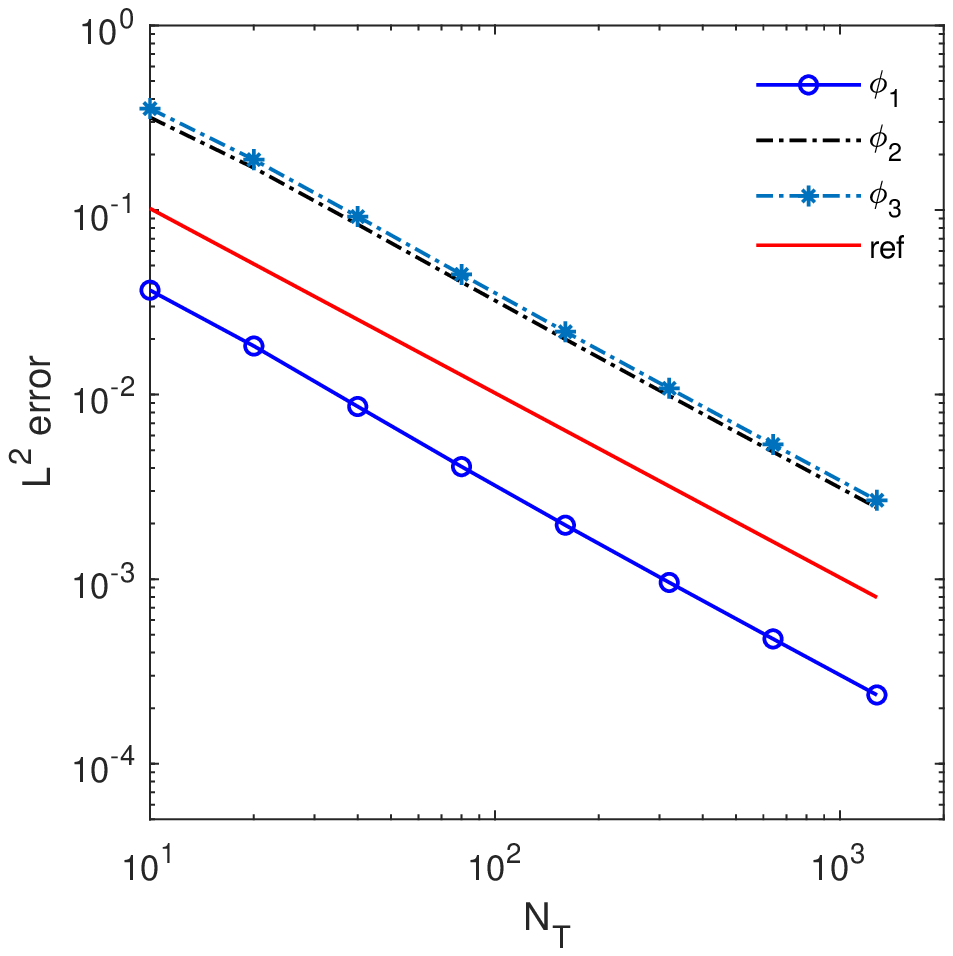}
			\caption{The  $L^\infty$ and $L^2$ numerical errors versus temporal resolution  $N_T$, at the final time $T=0.02$ in Example 6.1, by fixing $h=1/256$. The time step size is given by $\tau = \frac{T}{N_T}$. The reference line has an exact slope of -1, while the least square approximation to the $L^2$ error curves has approximate slopes -1.0466,  -1.0122, -1.0154, for the variables $\phi_1$, $\phi_2$ and $\phi_3$, respectively.}\label{fig_conv_t}
		\end{center}
	\end{figure}	

In the accuracy test for the spatial convergence order, we set the time size as $\tau= 7.8125 e-6$, so that the temporal error is negligible. A sequence of spatial resolutions are taken, with $h=\frac{1}{N_0}$.  The expected temporal numerical accuracy assumption $e=C h^2$ indicates that $\ln |e|=\ln C - 2 \ln N_0$, so that we plot $\ln |e|$ versus $\ln N_0$ to demonstrate the temporal convergence order. Similarly, Figure~\ref{fig_conv_h} presents the $L^{\infty}$ and $L^2$ numerical errors of the three-phase variables, as well as a reference line at the terminal time $T=0.02$, for this spatial convergence order test. The reference line has an exact slope of -2, while the least square approximation to the $L^2$ error curves has approximate slopes -2.0532,  -2.0476, -2.0480, for the variables $\phi_1$, $\phi_2$ and $\phi_3$, respectively. Therefore, a perfect second order spatial convergence rate is reported.    

	\begin{figure}[!htpb]
		\begin{center}
			\includegraphics[width=0.48\textwidth]{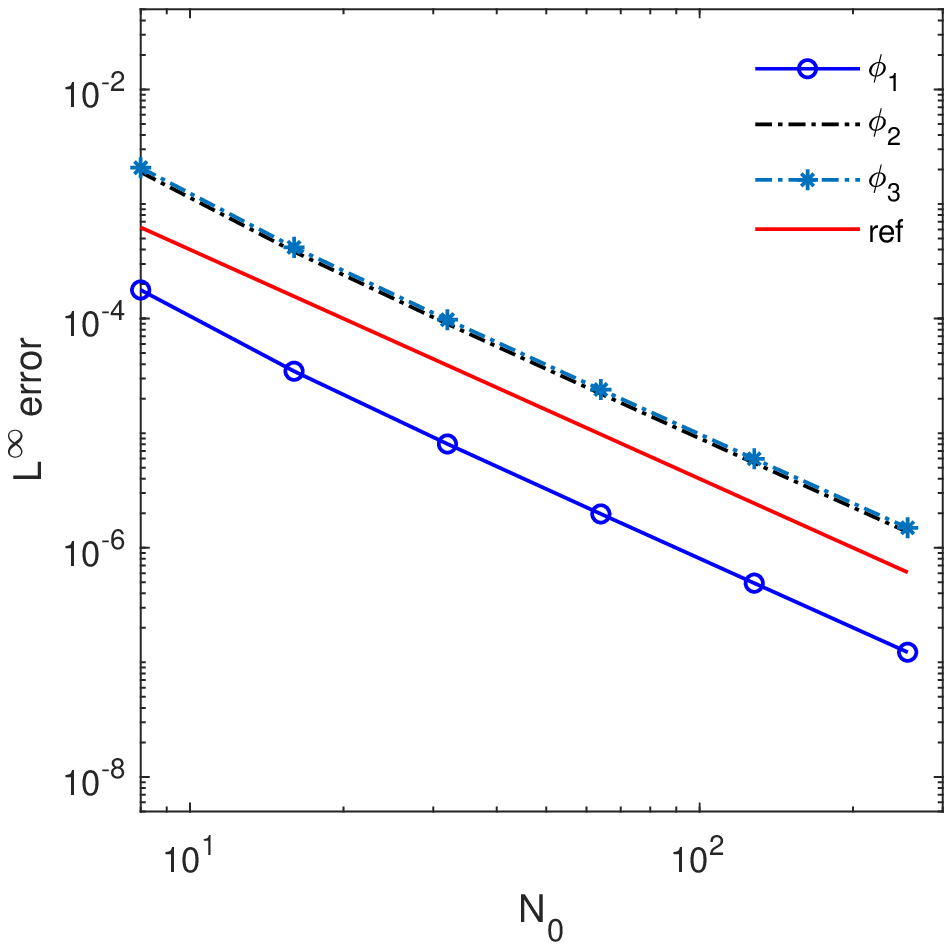}	\includegraphics[width=0.48\textwidth]{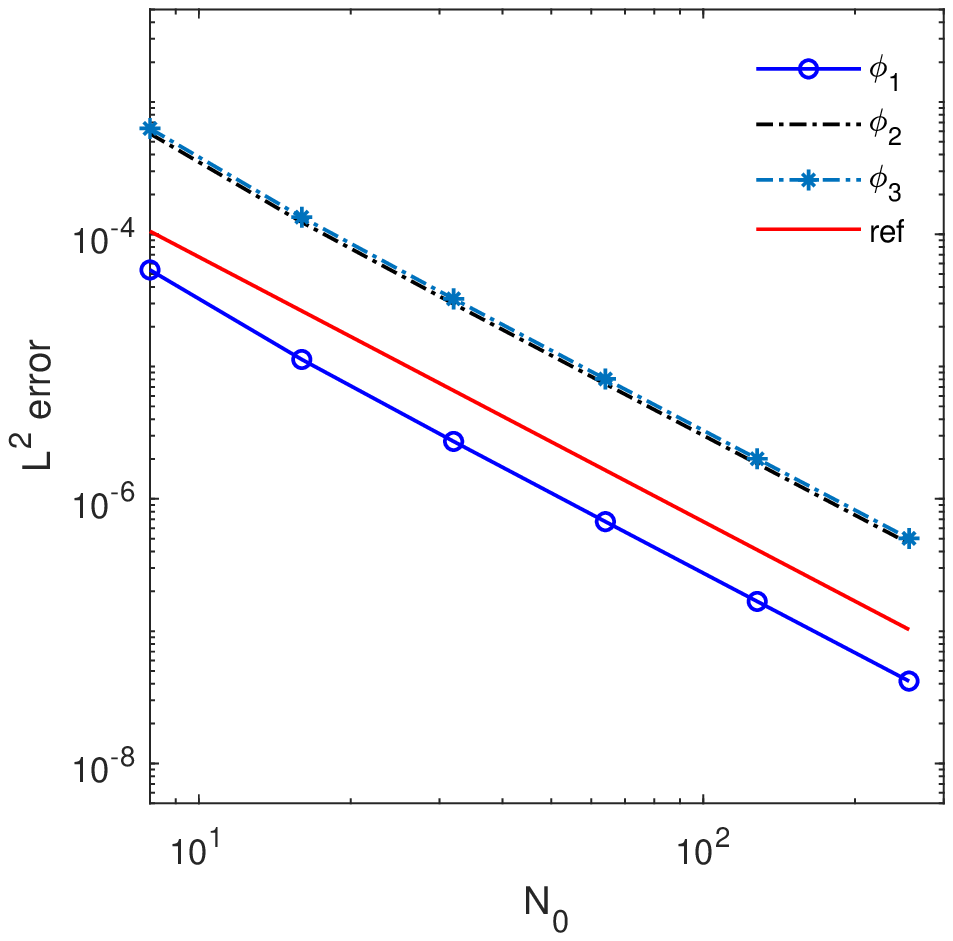}
			\caption{The  $L^\infty$ and $L^2$ numerical errors versus spatial resolution  $N_0$, at the final time $T=0.02$ in Example 6.1, by fixing $\tau= 7.8125 e-6$. The spatial mesh size is given by $h = \frac{1}{N_0}$. The reference line has an exact slope of -2, while the least square approximation to the $L^2$ error curves has approximate slopes -2.0532,  -2.0476, -2.0480, for the variables $\phi_1$, $\phi_2$ and $\phi_3$, respectively.}\label{fig_conv_h}
		\end{center}
	\end{figure}

\begin{example} 
	Consider the MMC-TDGL equation over the domain $\Omega = (0,64)^2$, with the initial data given by
	\begin{equation}\label{initial-2}
	\begin{aligned}
	\phi_{1}(x, y,0)&=0.1+0.01\cos(3\pi x/32)\cos(3\pi y/32),\\
	\phi_{2}(x, y,0)&=0.5+0.01\cos(3\pi x/32)\cos(3\pi y/32).
	\end{aligned}
	\end{equation}
\end{example}

We use the triangular mesh with size $h=1/4$ for partition of the domain, and take the time step size as $\tau=0.01$. Figure~\ref{fig_phase_evol2} displays the configuration of the simulated solution $\phi_2$ at a sequence of time instants, $t$=0, 5, 8, 10, 15 and 20, respectively. It is observed that the phase structures have a drastic change in time, and then asymptotically evolve to a steady state, which is consistent with the energy evolution plotted in Figure~\ref{fig_egy}. In addition, the configuration of all three phase variables are presented in Figure~\ref{fig_phase_evol}, at a sequence of later time instants, $t$=25, 80, and 200, respectively. 
The corresponding evolutions of the mass, as well as the maximum and minimum values of the phase variables, are displayed in Figures~\ref{fig_mass}, \ref{fig_max}, respectively. The mass conservation and the positivity property are observed to be preserved in these evolution figures.
\begin{figure}[!htpb]
	\begin{center}    		 
		\includegraphics[scale=0.8]{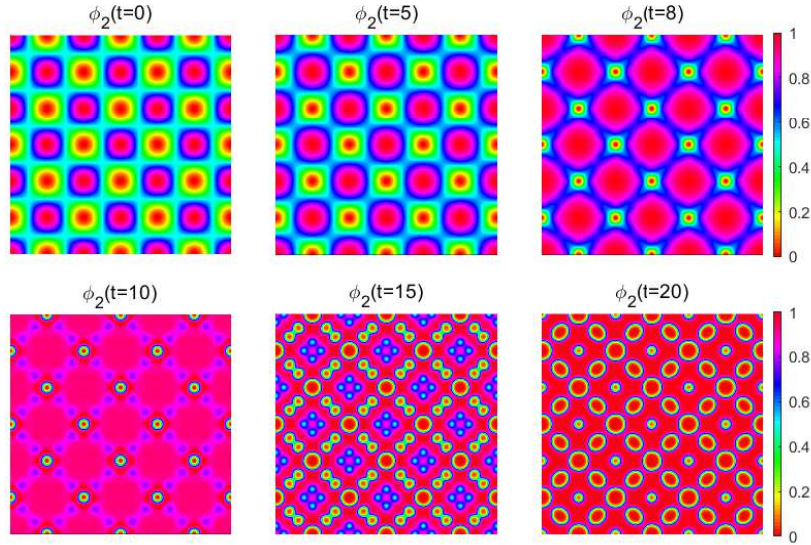}
		\caption{The simulated solution $\phi_2$, at $t$=0, 5, 8, 10, 15 and 20 respectively, in Example 6.2}\label{fig_phase_evol2}	
	\end{center}
\end{figure}
\begin{figure}[!htpb]
	\begin{center}		\includegraphics[scale=1]{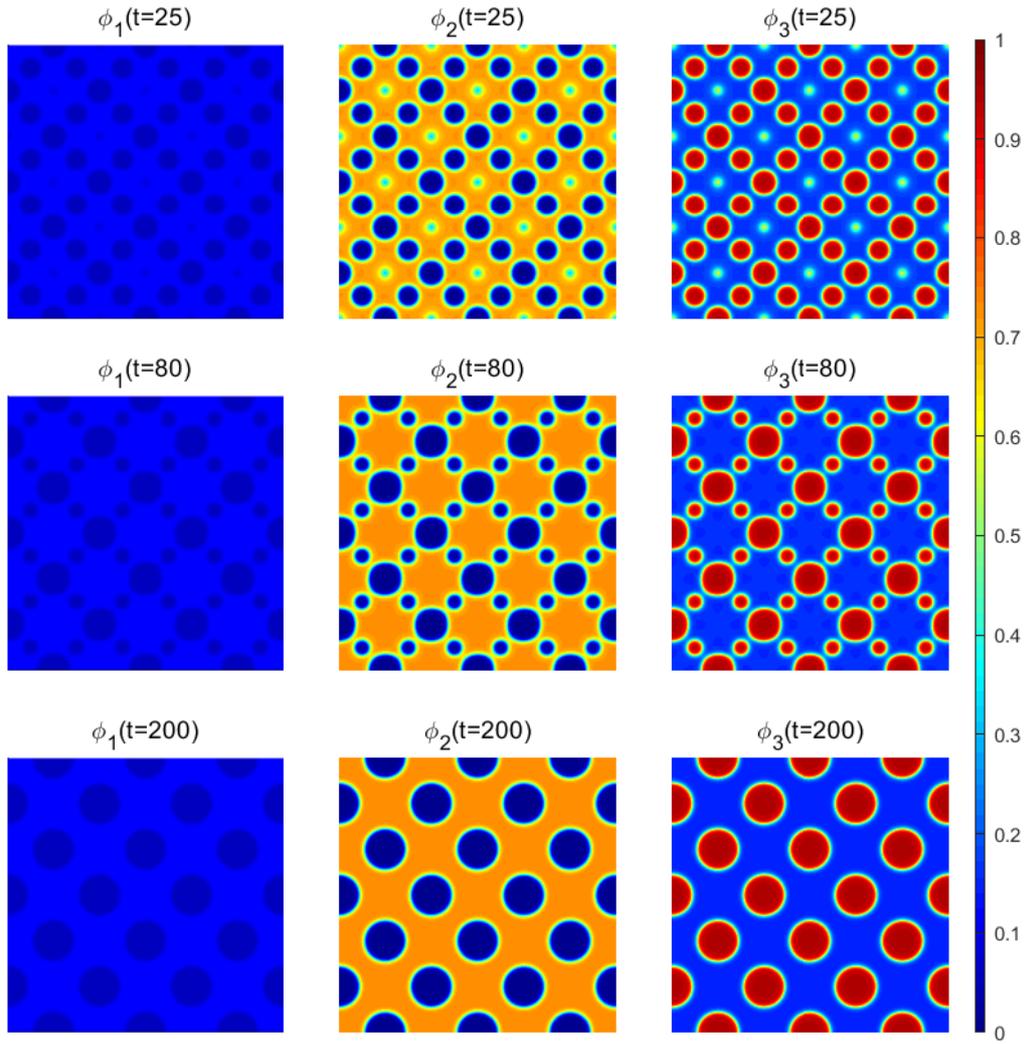}
		\caption{The phase variables plot in Example 6.2, at $t$=25, 80 and 200}\label{fig_phase_evol}	
	\end{center}
\end{figure}
\begin{figure}[!htpb]
	\begin{center}
		\includegraphics[scale=0.6]{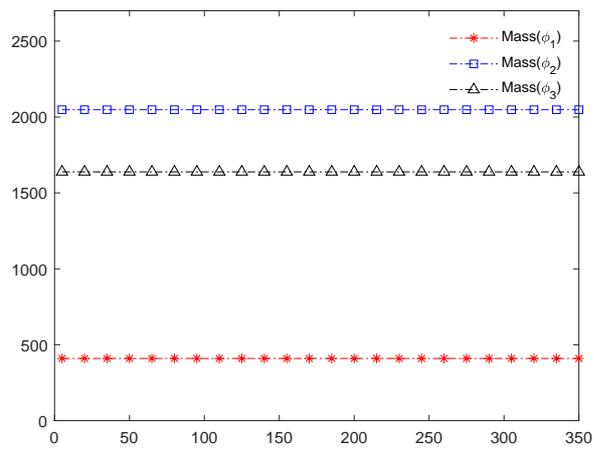}
		\caption{Mass evolution of the phase variables in Example 6.2}\label{fig_mass}
	\end{center}
\end{figure}
\begin{figure}[!htpb]
	\begin{center}
		\includegraphics[scale=0.8]{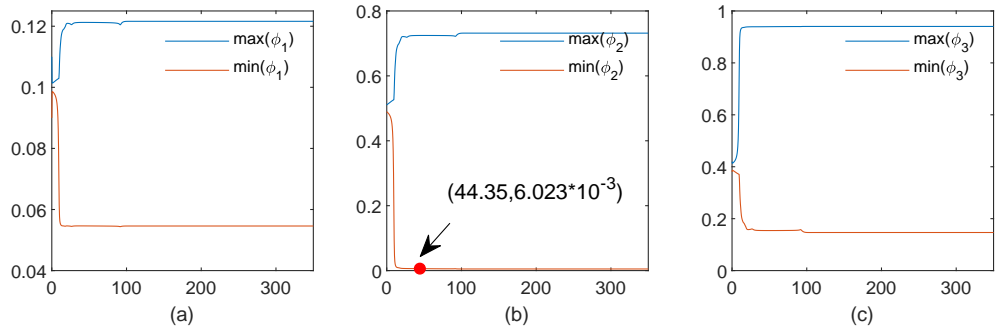}
		\caption{Evolution of the maximum and minimum value of phase variables in Example 6.2}
		\label{fig_max}
	\end{center}
\end{figure}
\begin{figure}[!htpb]
	\begin{center}
		\includegraphics[scale=0.8]{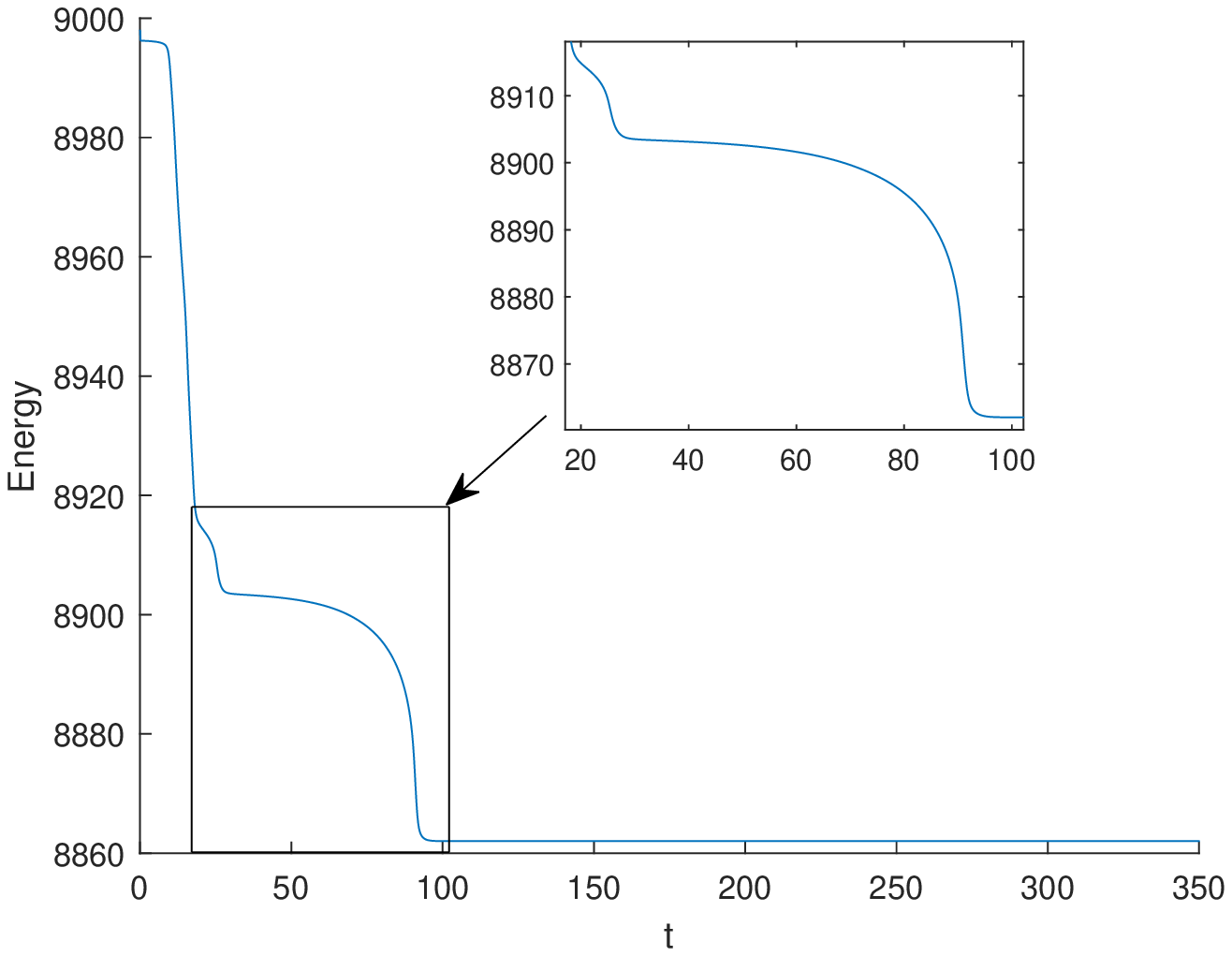}
		\caption{Energy evolution of the simulated solution  in Example 6.2}\label{fig_egy}
	\end{center}
\end{figure}

\begin{example}
Considered the MMC-TDGL equations over the domain $\varOmega = (0,50)^2$, with the initial data given by
	\begin{equation}
	\begin{aligned}
	\phi_{1}(x, y,0)&=0.1+r_{i,j},\\
	\phi_{2}(x, y,0)&=0.5+r_{i,j},
	\end{aligned}
	\end{equation}
	where the $r_{i,j}$ are uniformly distributed random numbers in $[-0.01,0.01]$.	
	\end{example}
	
	We use the uniform triangular mesh with size $h=1/4$, take the time step size as $\tau=0.01$, and focus on the $\phi_2$ variable, which reflects the polymer chain distribution. In this example, the initial concentration of polymer segments reaches $0.5 + r_{ i, j }$, every MMS can be joined by polymer chains since there are enough segments 
to grow. Thus the reticular structure can be obtained. Figure~\ref{fig_phase_rands} displays the plot of the $\phi_2$ variable at a sequence of time instants, $t$=0, 3.6, 6.52, 8, 10, 26, 85, 278 and 500, respectively. It is observed that the red area in the third row becomes larger, that is, the structure is tighter, which is consistent with \cite{li_phase_2015, ji2021modeling}.
\begin{figure}
	\begin{center}
		\includegraphics[scale=1]{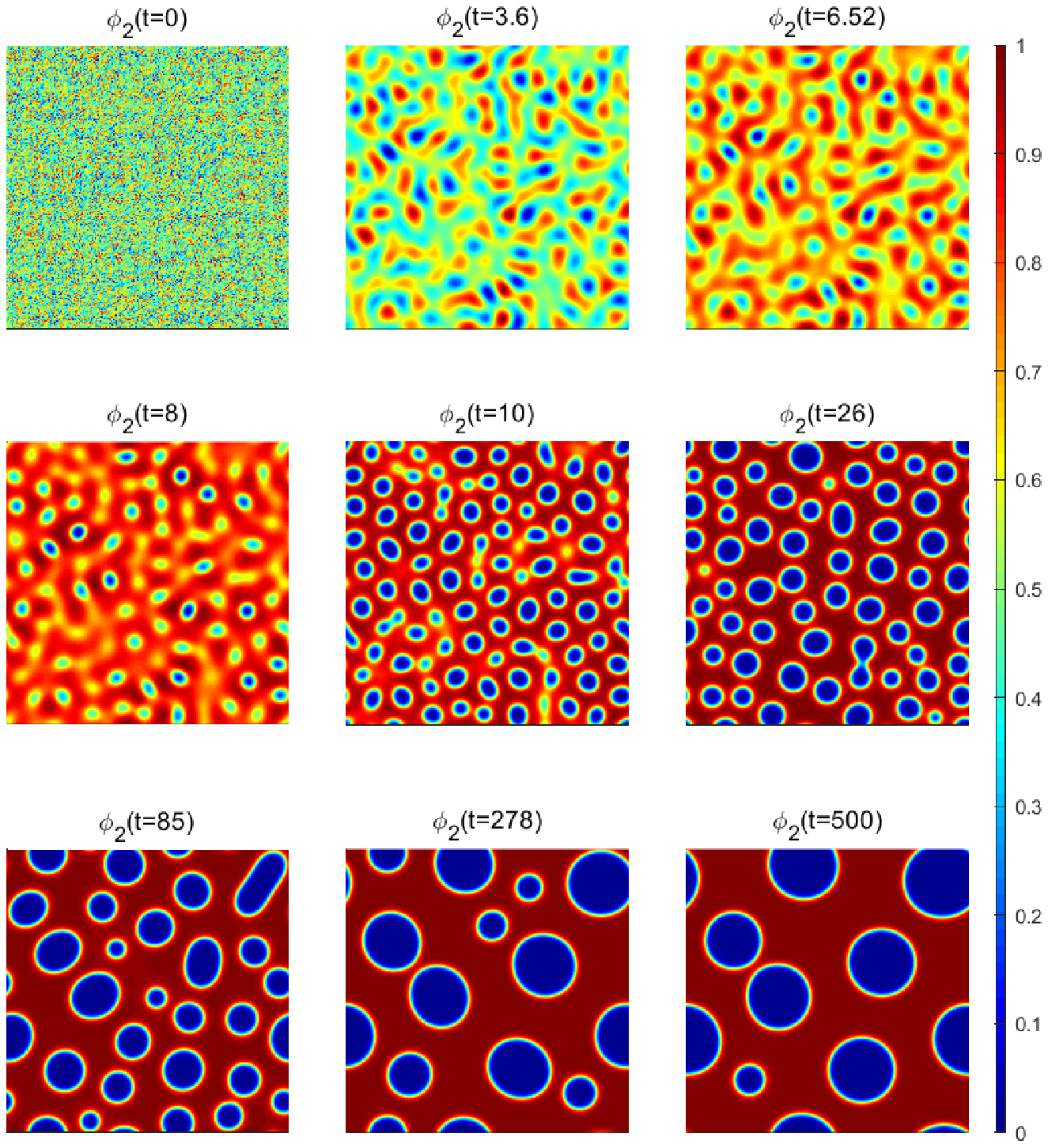}
		\caption{The phase variable plot for $\phi_2$ at a sequence of time instants $t$=0, 3.6, 6.52, 8, 10, 26, 85, 278 and 500 respectively in Example 6.3.}\label{fig_phase_rands}
	\end{center}
\end{figure}		
	
	 The evolution of the corresponding energy is plotted in Figure~\ref{fig_egy_rands}, which indicates a monotone decrease in time. Figures~\ref{fig_max_rands} and \ref{fig_mass_rands} display the maximum and minimum value of the phase variables and the mass. Again, the positivity-preserving property and mass conversation have been  perfectly demonstrated in the numerical simulation.
	\begin{figure}
	\begin{center}
		\includegraphics[scale=0.8]{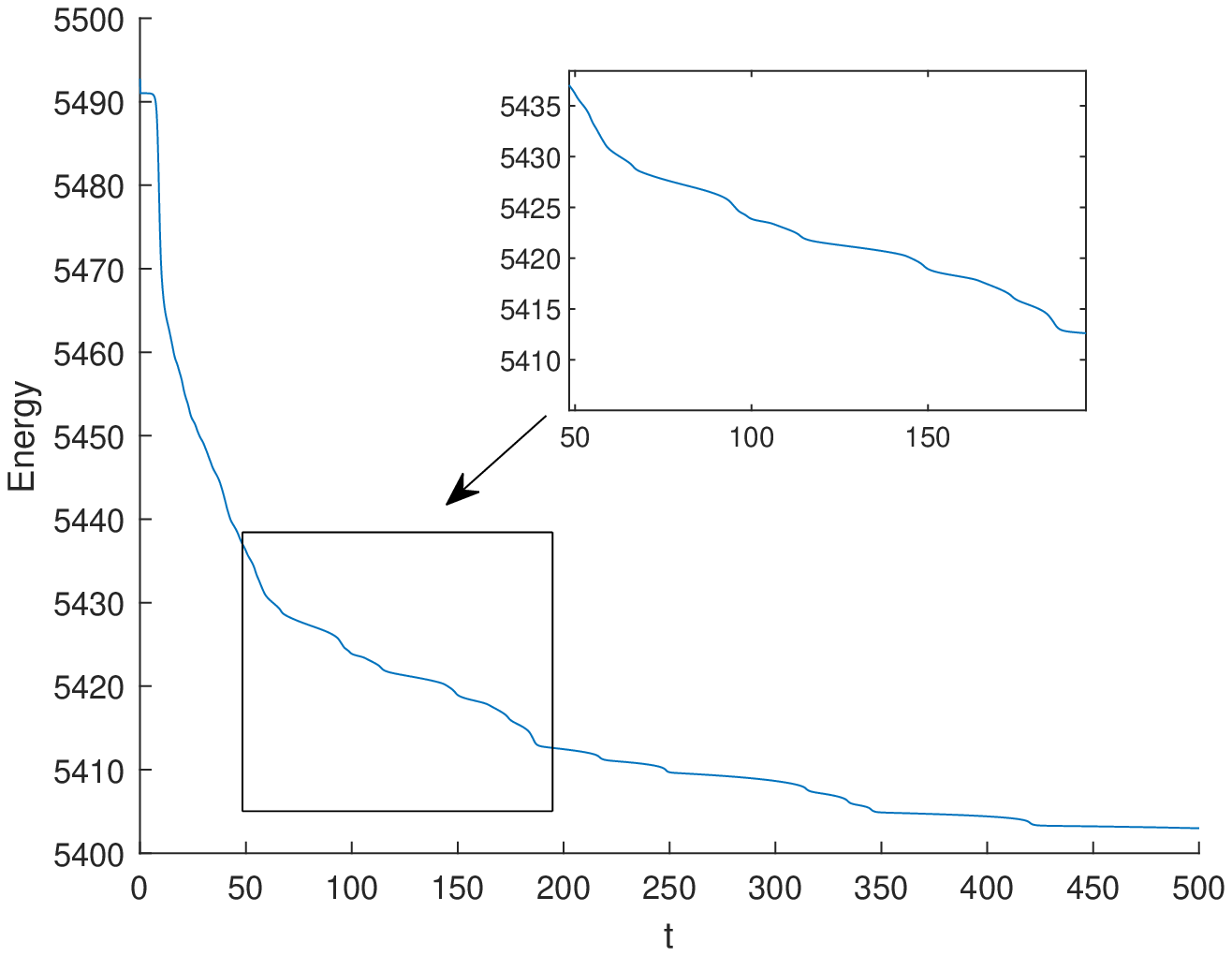}
		\caption{Energy evolution in Example 6.3.}\label{fig_egy_rands}
	\end{center}
\end{figure}
\begin{figure}
	\begin{center}
		\includegraphics[scale=0.8]{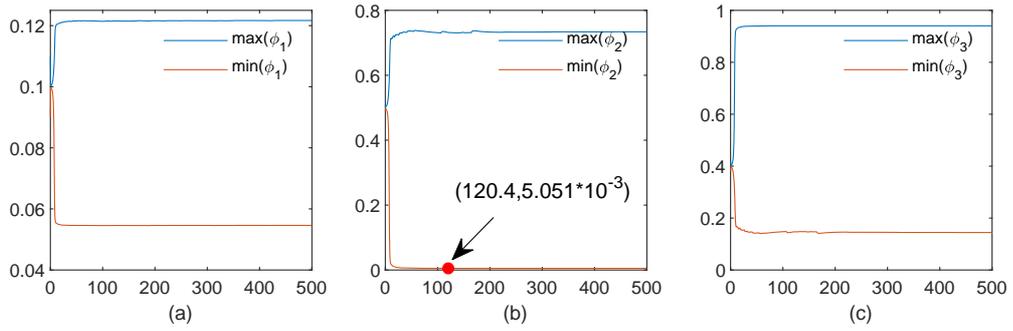}
		\caption{Evolution of the maximum and minimum values of the phase variables $\phi_1,\phi_2$ and $\phi_3$in Example 5.2.}\label{fig_max_rands}
	\end{center}
\end{figure}
\begin{figure}
	\begin{center}
		\includegraphics[scale=0.6]{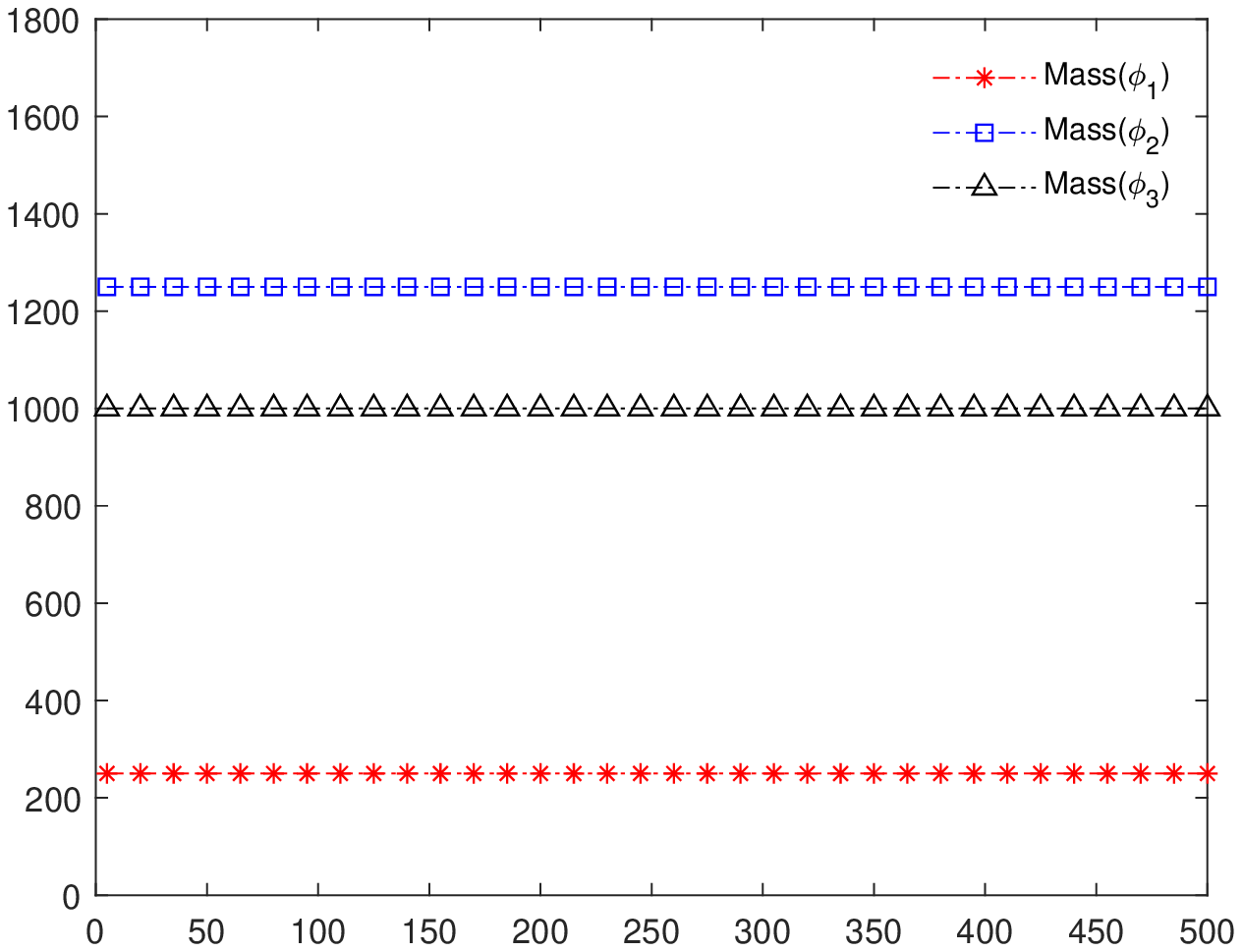}
		\caption{Mass evolution of phase variables in Example 6.3.} \label{fig_mass_rands}
	\end{center}
\end{figure}

\section{Concluding remarks} \label{sec:conclusion}

In this paper, we have developed a positivity-preserving and energy stable finite element scheme for the three-component Cahn-Hilliard flow model involved in macromolecular microsphere composite hydrogels, with the Flory-Huggins-deGennes energy potential in the ternary system. A convex-concave decomposition of the energy functional in multi-phase space is recalled, which in turn leads to an implicit treatment of the logarithmic and the nonlinear surface diffusion terms, as well as an explicit update of the concave expansive linear terms. In the spatial discretization, the mass lumped finite element approximation is applied. Both the positivity preserving property and the unconditional energy stability are theoretically justified, which will be the first such results for a finite element scheme applied to the ternary MMC system. A few numerical examples are presented, which demonstrate the robustness and accuracy of the proposed numerical scheme.

	\section*{Declarations} 
	
\noindent 	
{\bf Funding:} \,  W.B.~Chen is partially supported by the National Natural Science Foundation of China~(NSFC) 12071090, Shanghai Science and technology research program 19JC1420101 and a 111 project B08018.  Z.R.~Zhang is partially supported by NSFC No.11871105 and Science Challenge Project No. TZ2018002. C.~Wang is partially supported by the NSF DMS-2012669, S.M. Wise is partially supported by the NSF NSF-DMS 1719854, DMS-2012634.

\noindent
{\bf Conflicts of interest/Competing interests:} \, not applicable 

\noindent 
{\bf Availability of data and material:} \, not applicable 

\noindent
{\bf Code availability:} \, not applicable 

\noindent
{\bf Ethics approval:} \, not applicable 

\noindent
{\bf Consent to participate:} \, not applicable 

\noindent 
{\bf Consent for publication:} \, not applicable

\bibliographystyle{abbrv}
\bibliography{final}

\end{document}